\documentclass[11pt]{amsproc}
\usepackage{amscd, amsmath, amsthm, amssymb, mathtools, mathrsfs, stmaryrd}
\usepackage{ascmac}
\usepackage{setspace}
\usepackage{comment}
\usepackage{bm, bbm}
\allowdisplaybreaks

\usepackage{fullpage}

\usepackage{hyperref}

\usepackage{tikz}
\usetikzlibrary{intersections, calc, arrows.meta}

\usepackage{here}
\usepackage{time}
\usepackage[abbrev]{amsrefs}

\usepackage{xcolor}
\usepackage[capitalize,nameinlink,noabbrev,nosort]{cleveref}
\hypersetup{
	colorlinks=true,       
	linkcolor=brown,          
	citecolor=brown,        
	filecolor=brown,      
	urlcolor=brown,           
}

\makeatletter
\@namedef{subjclassname@2020}{\textup{2020} Mathematics Subject Classification}
\makeatother

\newtheorem{theoremcounter}{Theorem Counter}[section]

\theoremstyle{remark}

\theoremstyle{definition}
\newtheorem{definition}[theoremcounter]{Definition}

\theoremstyle{plain}
\newtheorem{lemma}[theoremcounter]{Lemma}
\newtheorem{proposition}[theoremcounter]{Proposition}
\newtheorem{corollary}[theoremcounter]{Corollary}

\newtheorem{theorem}[theoremcounter]{Theorem}

\numberwithin{equation}{section}

\newcommand{\Z}{\mathbb{Z}}
\newcommand{\Q}{\mathbb{Q}}
\newcommand{\R}{\mathbb{R}}
\newcommand{\C}{\mathbb{C}}

\newcommand{\calC}{\mathcal{C}}
\newcommand{\calD}{\mathcal{D}}
\newcommand{\calE}{\mathcal{E}}
\newcommand{\ep}{\varepsilon}

\newcommand{\dd}{\mathrm{d}}
\newcommand{\bbH}{\mathbb{H}}

\DeclareMathOperator{\ImNew}{Im}
\renewcommand{\Im}{\ImNew}

\DeclareMathOperator{\KW}{KW}

\newcommand{\pmat}[1]{\begin{pmatrix}#1\end{pmatrix}}

\newcommand{\smat}[1]{\bigl(\begin{smallmatrix}#1\end{smallmatrix}\bigr)}

\def\M{\mathrm{M}}
\def\gl{\mathfrak{gl}}
\def\sl{\mathfrak{sl}}
\def\spo{\mathfrak{spo}}

\def\sp{\mathfrak{sp}}
\def\so{\mathfrak{so}}

\def\t{{}^t\hspace{-.2em}}
\def\st{{}^{st}}
\def\largest{{\begin{matrix}\phantom{a} \\ \phantom{a}\end{matrix}}^{st\hspace{-.4em}}}
\def\1{\mathbbm{1}}

\DeclareMathOperator{\diag}{diag} 
\DeclareMathOperator{\tr}{tr}
\DeclareMathOperator{\str}{str}

\DeclareMathOperator{\Aut}{Aut}

\DeclareMathOperator{\sgn}{sgn}

\newcommand{\fg}{\mathfrak{g}}
\newcommand{\fh}{\mathfrak{h}}

\newcommand{\fq}{\mathfrak{q}}

\newcommand{\fS}{\mathfrak{S}}

\newcommand{\bP}{\mathbb{P}}

\begin{document}

\title[]{Denominator identity for the affine Lie superalgebra $\widehat{\mathfrak{spo}}(2m,2m+1)$ and indefinite theta functions}

\author{Toshiki Matsusaka}
\address{Faculty of Mathematics, Kyushu University, Motooka 744, Nishi-ku, Fukuoka 819-0395, Japan}
\email{matsusaka@math.kyushu-u.ac.jp}

\author[]{Miyu Suzuki}
\address{Department of Mathematics, Kyoto University, Kitashirakawa Oiwake-cho, Sakyo-ku,  Kyoto 606-8502, Japan}
\email{suzuki.miyu.4c@kyoto-u.ac.jp}


\subjclass[2020]{11F27,  17B10}



\maketitle

\begin{abstract}
In 1994,  Kac and Wakimoto found the denominator identity for classical affine Lie superalgebras,  generalizing that for affine Lie algebras.
As an application,  they obtained power series identities for some powers of $\triangle(q)$,  where $\triangle(q)$ is the generating function of triangular numbers.
In this article,  we give a different proof of one of their identities.
The main step is to prove that a certain indefinite theta function involving spherical polynomials is a modular form. 
We use the technique recently developed by Roehrig and Zwegers. 
\end{abstract}


\section{Introduction}

The \emph{denominator identity} of a finite-dimensional semisimple Lie algebra $\fg$ is an identity of the following form:
    \begin{equation}\label{intro1}
    \prod_{\alpha\in\Phi^+}(1-e^{-\alpha}) 
    = \sum_{w\in W}\epsilon(w)e^{w(\rho)-\rho}.
    \end{equation}
Here,  the product on the left-hand side is over the positive roots of $\fg$,  the sum in the right-hand side is over the Weyl group,  $\epsilon$ is the sign character,  and $\rho$ is the Weyl vector.
The equation is understood in a certain ring of Laurent polynomials in $e^{-\alpha}$'s.
In particular,  both sides of \eqref{intro1} are finite.
For example,  when $\fg=\sl(n)$,  the product becomes the Vandermonde polynomial and the sum becomes the determinant of the Vandermonde matrix.

The denominator identity is related to the representation theory of Lie algebras. 
It follows from the Weyl character formula for highest weight modules of $\fg$,   applied to the trivial representation.
The left-hand side of \eqref{intro1} appears in the denominator of the character formula and is called the \emph{denominator} of $\fg$.

In the 1970's, Kac and Moody established the structure and representation theory of a certain class of infinite-dimensional Lie algebras $\fg$,  which are called the Kac--Moody Lie algebras.
They extended the Weyl character formula to integrable highest weight modules when $\fg$ is symmetrizable.
Similarly as in the finite-dimensional case,  the denominator identity for $\fg$ is obtained.
It is of exactly the same form as \eqref{intro1} except that both the product and the sum become infinite.
For details,  see the original articles of Moody~\cite{Moody},  Kac~\cite{Kac2},  and Chapter 12 of the book of Kac~\cite{Kac}.

A typical example of a Kac--Moody Lie algebra is an affine Lie algebra.
When $\fg$ is an affine Lie algebra,  the denominator identity is called the Macdonald identity.
This is because it was first established by Macdonald \cite{Macdonald} without using Lie algebras.
Substituting appropriate variables to $e^{-\alpha}$'s and $e^{-\rho}$,  one can deduce from the Macdonald identity a large number of famous identities including the Jacobi triple product identity and the Watson quintuple product identity.

A Lie superalgebra is a vast generalization of a Lie algebra. 
It stems from the theory of supersymmetry in mathematical physics and has been studied with representation theoretic interest. 
In 1994,  Kac and Wakimoto \cite{KacWakimoto1994} formulated the following denominator identity for a classical affine Lie superalgebra $\fg$.
The proof was completed by Gorelik \cite{Gorelik11} in 2011:
    \begin{equation}\label{intro2}
    \frac{\prod_{\alpha\in\widehat{\Phi}_0^+}(1-e^{-\alpha})}
    {\prod_{\alpha\in\widehat{\Phi}_1^+}(1+e^{-\alpha})}
    = \sum_{w\in\widehat{W}^\sharp}\epsilon(w)\frac{e^{w(\widehat{\rho})-\widehat{\rho}}}
    {\prod_{\beta\in S}(1+e^{-w(\beta)})}.
    \end{equation}
The left-hand side is a quotient of two infinite products and the right-hand side is an infinite sum.
The details of this formula are exlpained in \cref{Affine Lie superalgebra} in the case $\fg=\widehat{\spo}(2m,  2m+1)$.

The subject of Kac--Wakimoto \cite{KacWakimoto1994} is to deduce some identities of $q$-series from \eqref{intro2},  which are of interest in combinatorics and number theory.
Let $\triangle(q)$ be the generating function of triangular numbers, \emph{i.e.}
    \[
    \triangle(q) \coloneqq \sum_{n=0}^\infty q^{\frac{n(n+1)}{2}}.
    \]
For $N,  r\in\Z_{>0}$,  the coefficient of $q^N$ in $\triangle(q)^r$ equals the number of representations of $N$ as a sum of $r$ triangular numbers.
Kac--Wakimoto especially studied the case of $\fg=\widehat{\sl}(m+1,  m)$,  $\widehat{\spo}(2m,  2m+1)$,  and $\widehat{\fq}(m)$ to obtain power series expansions of $\triangle(q)^r$.

When $\fg=\widehat{\spo}(2m,  2m+1)$,  the resulting identity is as follows.
\begin{theorem}[{Kac--Wakimoto~\cite{KacWakimoto1994}}]\label{thm:Kac-Wakimoto-Conj}
For a positive integer $m$,  we have
    \begin{align}\label{intro3}
    \begin{split}
	\triangle(q)^{2m(2m+1)} 
	&= \frac{1}{m! \left(\prod_{j=1}^m (2j-1)!\right)^2} 
	\sum_{\substack{k_1, \dots, k_m \ge 0 \\ r_1, \dots, r_m \ge 0}} 
	(-1)^{k_1+\cdots + k_m}  \\  
	&\quad \times \prod_{1 \le i < j \le m} (r_i + k_i - r_j - k_j)(r_i - r_j)
	(2+r_i+k_i+r_j+k_j)(1+r_i+r_j)  \\  
	&\quad \times \prod_{j=1}^m (1+2r_j)(1+r_j+k_j) 
	q^{\sum_{j=1}^m (k_j^2 + 2k_j r_j + 2k_j + r_j) - \frac{m(m-1)}{2}}.
	\end{split}
    \end{align}
\end{theorem}

The idea of the proof is simple.
Divide the left-hand side of \eqref{intro2} by $R_0\coloneqq\prod_{\alpha\in\Phi_0^+}(1-e^{-\alpha})$ and substitute $-q$ to $e^{-\delta/2}$,  where $\delta\in\widehat{\Phi}_0^+$ is the null root.
Finally,  take the limit as $e^{-\alpha}\to1$ for each finite simple root $\alpha$,  then the infinite product becomes $\left(\prod_{n=1}^\infty(1-q^{2n})(1-q^{2n-1})^{-1}\right)^{2m(2m+1)}$.
By the following well-known identity,  we obtain $\triangle(q)^{2m(2m+1)}$ from the left-hand side: 
    \begin{equation}\label{eq:triangle}
    \triangle(q) = \prod_{n=1}^\infty\frac{1-q^{2n}}{1-q^{2n-1}}.
    \end{equation}
    
Unfortunately,  the computation in Kac--Wakimoto \cite{KacWakimoto1994} is slightly incorrect,  especially the formula in their Example 5.4 needs one minor modification.
As a matter of fact,  they knew the correct formula \eqref{intro3},  which is stated as Conjecture 5.1 in their article.
\textbf{The first goal of this article is to see that their conjecture is a consequence of their computation. }
    
The equation \eqref{eq:triangle} is an easy consequence of the Jacobi triple product identity. 
Another way to look at \eqref{eq:triangle} is to regard both sides as \emph{modular forms}.
After multiplying by $q^{1/8}$,  the left-hand side becomes the Jacobi theta function and the right-hand side becomes a certain quotient of the Dedekind eta function. 
It is known that both are modular forms of weight $1/2$ on $\Gamma(2)$. 
See \eqref{eq:theta-eta} for details.
In particular,  this observation implies that the left-hand side of \eqref{intro3},  multiplied by a certain power of $q$,  is a modular form and so is the right-hand side. 

Conversely,  suppose that one can prove that both sides of \eqref{intro3} are modular forms with the same weight,  the same level,  and the same asymptotic behavior at each cusp.
Then the quotient $\triangle(q)^{-2m(2m+1)} \times (\mathrm{RHS}\text{ of \eqref{intro3}})$ defines a holomorphic function on $\bbH^\ast/\Gamma(2)\simeq\bP^1(\C)$,  where $\bbH^\ast \coloneqq \bbH\cup\Q\cup\{i\infty\}$ is the completed upper half-plane.
Hence it should be a constant,  which turns out to equal $1$.
\textbf{The second goal of this article is to give another proof of \cref{thm:Kac-Wakimoto-Conj} in this way.} 
The strategy is as follows: The expression in~\cref{thm:Kac-Wakimoto-Conj} may appear intricate, but through an appropriate change of variables, it transforms into a form resembling a theta function associated with an indefinite quadratic from, as shown in~\cref{thm:KW-rewrite}. 
In particular, for $m=1$, by entering the theory of indefinite theta functions with spherical polynomials, established by Roehrig--Zwegers~\cite{RoehrigZwegers2022, RoehrigZwegers2022-pre}, its modularity follows. 
The aim in the latter part of this article is to extend their result to general $m$ by formulating theorems (\cref{thm:indefinite-theta} and \cref{thm:KW-indefinite}) that apply in general cases.

There are several other affine Lie superalgebras whose denominator identities provide equations for some powers of $\triangle(q)$.
Zagier \cite{Zagier2000} studied the case of $\triangle(q)^{4m^2}$ and $\triangle(q)^{4m(m+1)}$ which are obtained from the denominator identities of the affine Lie superalgebra $\widehat{\mathfrak{q}}(N)$.
In \cite{MatsusakaSuzuki2025-pre}, the authors study the remaining cases with a similar method.

This article is organized as follows.  
The first two sections are the algebraic part.
In \cref{Lie-superalgebra},  we introduce the finite-dimensional Lie superalgebra $\fg=\spo(2m,  2m+1)$ and the notation for the root system of $\fg$.
\cref{Affine Lie superalgebra} is devoted to following the computation of Kac--Wakimoto \cite{KacWakimoto1994}.
We introduce the affine Lie superalgebra $\widehat{\fg}=\widehat{\spo}(2m,  2m+1)$ and explicitly write down its root system and the denominator identity. 

The latter two sections are the analytic part.
Vign\'eras' result (\cref{fact:Vig}) on the modular transformation of theta functions plays a crucial role in our argument. 
We rewrite the right-hand side of \eqref{intro3} as an indefinite theta function $\mathrm{KW}_m(\tau)$ in \cref{Modular-framework},  so that Vign\'eras' theorem can be applied.
To apply Vign\'eras' theorem,  we need to find an appropriate function $p\,\colon\R^{2m}\to\C$ with two important properties.
In \cref{indefinite-theta},  we define the function $p=p^{\vec{\bm{c}}_0, \vec{\bm{c}}_1}[f](\vec{\bm{x}})$ generalizing the construction of Roehrig--Zwegers and prove that it has the necessary properties if the polynomial $f$ is spherical.
We show in \cref{modularity-KW} that the polynomial $f=V_m(\vec{\bm{x}})$,  which appears in $\mathrm{KW}_m(\tau)$,  is a spherical polynomial.
After a careful argument on some convergence problems,  we obtain the modular transformation laws for $\mathrm{KW}_m(\tau)$ (\cref{cor:modular-KW}) and its asymptotic behavior at each cusp (\cref{cor:cusps-KW}).
Combining all these results,  we obtain a new proof of \cref{thm:Kac-Wakimoto-Conj}.

\bigskip

\noindent\textbf{Data Availability Statement.}
All data generated or analyzed during this study are included in this published article.

\bigskip

\noindent\textbf{Compliance with Ethical Standards.}
All authors declare that they have no conflicts of interest.

\section*{Acknowledgment}

The authors would like to thank Kazuki Kannaka for his valuable comments during the closed study group on affine Lie superalgebras. 
They also thank the anonymous referees for their careful reading and helpful comments.
The first author was supported by JSPS KAKENHI (JP21K18141 and JP24K16901) and the MEXT Initiative through Kyushu University's Diversity and Super Global Training Program for Female and Young Faculty (SENTAN-Q). The second author was supported by JSPS KAKENHI (JP22K13891 and JP23K20785).

\section{Lie superalgebra $\spo(2m,  2m+1)$}\label{Lie-superalgebra}

Let $\fg=\fg_0\oplus\fg_1$ be a $\Z/2\Z$-graded vector space over $\C$ and $[\cdot ,\cdot ]\,\colon \fg\times\fg\to\fg$ a bilinear form.
We say that $(\fg,  [\cdot ,\cdot ])$ is a \emph{Lie superalgebra}  if it satisfies the following conditions (i)--(iii):
\begin{itemize}
\item[(i)] $[\fg_i,  \fg_j]\subset\fg_{i+j}$ for $i,  j\in\Z/2\Z$.
\item[(ii)] $[X, Y]+(-1)^{ij}[Y,X]=0$ for $X\in\fg_i$ and $Y\in\fg_j$.
\item[(iii)] $(-1)^{ki}[X,[Y,Z]]+(-1)^{ij}[Y,[Z,X]]+(-1)^{jk}[Z,[X,Y]]=0$ for $X\in\fg_i$,  $Y\in\fg_j$,  and $Z\in\fg_k$.
\end{itemize}
A graded subspace of $\fg$ is called a subalgebra if it is closed under the Lie bracket $[\cdot, \cdot]$.

A typical example is $\gl(M,  N)$ for non-negative integers $M,  N$.
As a vector space,  $\gl(M,  N)=\M_{M+N}(\C)$,  the grading is given by 
    \begin{align*}
    \gl(M,  N)_0 &= \left\{
        \begin{pmatrix}
        A &  \\
         & D
        \end{pmatrix}
    \,\middle|\,  A\in\M_M(\C),  \ D\in\M_N(\C) \right\},  \\
    \gl(M,  N)_1 &= \left\{
        \begin{pmatrix}
         & B \\
        C & 
        \end{pmatrix}
    \,\middle|\,  B\in\M_{M\times N}(\C),  \ C\in\M_{N\times M}(\C) \right\},
    \end{align*}
and the Lie bracket is determined by $[X,  Y]=XY-(-1)^{ij}YX$ for $X\in\gl(M,N)_i$ and $Y\in\gl(M,N)_j$.
We define the \emph{super-trace} $\str\,\colon\gl(M,N)\to\C$ by $\str
    \begin{pmatrix}
    A & B \\
    C & D
    \end{pmatrix}
=\tr(A)-\tr(D)$,  where $\tr(\cdot)$ denotes the usual trace of matrices.
Then,  the subspace $\sl(M, N)$ of $\gl(M,  N)$ consisting of $X\in\gl(M,  N)$ with $\str(X)=0$ is a subalgebra.

Let $(\cdot,  \cdot)\,\colon\gl(M,N)\times\gl(M,N)\to\C$ be a bilinear form given by $(X,  Y)=\str(XY)$,  where $XY$ denotes the usual product of matrices.
It is easy to check that this is non-degenerate and has the following properties:
\begin{itemize}
\item (\emph{even}): \ $(X,Y)=0$ for $X\in\gl(M,N)_i$ and $Y\in\gl(M,N)_j$ if $i+j=1$.
\item (\emph{invariant}): \ $([X,Y],  Z)=(X,[Y,Z])$ for $X,  Y,  Z\in\gl(M,  N)$.
\item (\emph{super-symmetric}): \ $(X,Y)=(-1)^{ij}(Y,  X)$ for $X\in\gl(M,N)_i$ and $Y\in\gl(M,N)_j$.
\end{itemize}

Let $\fg=\spo(2m,  2m+1)$ be the subalgebra of $\sl(2m,  2m+1)$ defined as
    \[
    \spo(2m,  2m+1) = \left\{X\in\sl(2m,  2m+1) \, \middle|\, \st X
         \begin{pmatrix}
        J_1 & \\
        & J_2
        \end{pmatrix}
    +
         \begin{pmatrix}
        J_1 & \\
        & J_2
        \end{pmatrix}
    X=0\right\}.
    \]
Here,   $J_1=
    \begin{pmatrix}
     & \1_m \\
    -\1_m & 
    \end{pmatrix}$,  $J_2=
    \begin{pmatrix}
     & & \1_m \\
     & 1 & \\
    \1_m & &
    \end{pmatrix}$
and $\st\bullet$ denotes the \emph{super-transpose} given by
    \[
    \largest
        \begin{pmatrix}
        A & B \\
        C & D
        \end{pmatrix}
    =
        \begin{pmatrix}
        \t A & \t C \\
        -\t B & \t D
        \end{pmatrix},  \qquad 
        \begin{pmatrix}
        A & B \\
        C & D
        \end{pmatrix}
    \in\gl(2m,  2m+1).
    \]
Then $\fg$ is called the \emph{ortho-symplectic Lie superalgebra} of type $B(m,m)$.
It consists of $
    \begin{pmatrix}
     A & B \\
    C & D
    \end{pmatrix}
\in\gl(2m,  2m+1)$ such that
    \begin{align*}
    & A\in\sp(2m) \coloneqq \{X\in\M_{2m}(\C) \mid \t XJ_1+J_1X=0\},   \\
    & D\in\so(2m+1) \coloneqq \{X\in\M_{2m+1}(\C) \mid \t XJ_2+J_2X=0\},   \\
    & B\in\M_{2m\times(2m+1)}(\C) ,  \quad C=J_2^{-1}\t BJ_1.
    \end{align*}  
Note that $\fg_0$ is a Lie algebra isomorphic to $\sp(2m)\oplus\so(2m+1)$.
One can check that the restriction of the bilinear form $(\cdot,  \cdot)$ on $\gl(2m,  2m+1)$ to $\spo(2m,  2m+1)$ is still non-degenerate.

Let $\fh$ be the \emph{Cartan subalgebra} of $\fg$ consisting of all diagonal elements.
A general element of $\fh$ is of the form
    \begin{equation}\label{Cartan-element}
    H = \diag(x_1,  \ldots,  x_m,  -x_1,  \ldots,  -x_m,  
    y_1,  \ldots,  y_m,  0,  -y_1,  \ldots,  -y_m)
    \end{equation}
with $x_i,  y_j\in\C$.
Since the restriction of $(\cdot,  \cdot)$ defines a non-degenerate pairing on $\fh$,  it induces a linear isomorphism $\fh^\ast\xrightarrow{\sim}\fh$,  where $\fh^\ast$ is the linear dual of $\fh$.
We define the pairing on $\fh^\ast$ as the pull-back of that on $\fh$ by this isomorphism,  which we again write as $(\cdot,  \cdot)$.

We define $\varepsilon_i,  \delta_j\in\fh^\ast$ by
    \[
    \varepsilon_i(H) = x_i,  \quad
    \delta_j(H) = y_j,  \hspace{30pt}  i,  j=1,  \ldots,  m,
    \]
where $H\in\fh$ is of the form \eqref{Cartan-element}.
Then,  the root system of $\fg$ with respect to $\fh$ is given as $\Phi=\Phi_0\cup\Phi_1$,  where
    \begin{align*}
    \Phi_0 & \coloneqq \{\pm(\varepsilon_i-\varepsilon_j), \ \pm(\varepsilon_i+\varepsilon_j), 
    \ \pm2\varepsilon_p \mid 1\leq i<j\leq m, \ 1\leq p\leq m\} \\
    & \qquad \cup\{\pm(\delta_i-\delta_j), \ \pm(\delta_i+\delta_j), \ 
    \pm\delta_q \mid 1\leq i<j\leq m, \ 1\leq q\leq m\},   \\
    \Phi_1 & \coloneqq \{\pm(\varepsilon_i-\delta_j), \ 
    \pm(\varepsilon_r+\delta_s), \ \pm\varepsilon_p 
    \mid 1\leq i,  j,  r,  s,  p\leq m\}.
    \end{align*}
For each root $\alpha\in\Phi$,  let $\fg_\alpha=\{X\in\fg \mid [H,  X]=\alpha(H)X \text{ for all $H\in\fh$}\}$ be the corresponding root subspace.
Then we have the root space decomposition $\fg_0 = \fh\oplus \bigoplus_{\alpha\in\Phi_0} \fg_\alpha$ and $\fg_1 = \bigoplus_{\alpha\in\Phi_1} \fg_\alpha$.

The \emph{Weyl group} $W$ of $\fg$ is defined to be the Weyl group of $\Phi_0$.
Since $\Phi_0$ is the union of two root systems of type $C_m$ and $B_m$,  we have $W\simeq(\fS_m\ltimes(\Z/2\Z)^m)\times(\fS_m\ltimes(\Z/2\Z)^m)$.

The following lemma is a consequence of an easy calculation.
\begin{lemma}
The set $\{\varepsilon_i,  \delta_j \mid 1\leq i,j\leq m\}$ forms an orthogonal basis of $\fh^\ast$ and we have $(\varepsilon_i,  \varepsilon_i)=1/2=-(\delta_j,  \delta_j)$ for $1\leq i,j\leq m$.

In particular,  we have
\begin{itemize}
\item $(\varepsilon_i\pm\varepsilon_j,  \varepsilon_i\pm\varepsilon_j)=1=-(\delta_i\pm\delta_j,  \delta_i\pm\delta_j)$,  
\item $(2\varepsilon_i,  2\varepsilon_i)=2$,
\item $(\varepsilon_i\pm\delta_j,  \varepsilon_i\pm\delta_j)=0$.
\end{itemize}
\end{lemma}

\begin{definition}
For $i=1,  \ldots,  2m$,  set
    \[
    \alpha_i=
        \begin{cases}
        \delta_k-\varepsilon_k & i=2k-1,  \ (1\leq k\leq m),  \\
        \varepsilon_k-\delta_{k+1} & i=2k,  \ (1\leq k\leq m-1),  \\
        \varepsilon_m & i=2m,
        \end{cases}
    \]
and $\Delta=\{\alpha_i \mid i=1,  \ldots,  2m\}$.
\end{definition}

Put $\Phi_i^+=(\sum_{\alpha\in\Delta}\Z_{\geq0}\cdot\alpha)\cap\Phi_i$,  ($i=0,  1$), and $\Phi^+=\Phi_0^+\cup\Phi_1^+$.
A simple calculation shows that
    \begin{align*}
    \Phi_0^+ & = \{\varepsilon_i-\varepsilon_j, \ 
    \varepsilon_i+\varepsilon_j, 
    \ 2\varepsilon_p \mid 1\leq i<j\leq m, \ 1\leq p\leq m\} \\
    & \qquad \cup\{\delta_i-\delta_j, \ \delta_i+\delta_j,  \ 
    \delta_q \mid 1\leq i<j\leq m, \ 1\leq q\leq m\},   \\
    \Phi_1^+ &= \{\delta_i-\varepsilon_j, \ 
    \varepsilon_k-\delta_l, \  \varepsilon_r+\delta_s, \ \varepsilon_p 
    \mid 1\leq i\leq j\leq m,  \ 1\le k<l\leq m,  \ 1\leq p,  r,  s\leq m \}.
    \end{align*}
In particular,  we have $\Phi_i=\Phi_i^+\cup(-\Phi_i^+)$, ($i=0,1$),  which implies that $\Delta$ is a fundamental set of roots.
 
Set    
    \begin{align*}
    \Phi_0^\sharp & = \{\alpha\in\Phi_0 \mid (\alpha,  \alpha)>0 \} \\
    & = \{\pm(\varepsilon_i-\varepsilon_j), \ \pm(\varepsilon_i+\varepsilon_j), 
    \ \pm2\varepsilon_p \mid 1\leq i<j\leq m, \ 1\leq p\leq m\}. 
    \end{align*}
This is a root system of type $C_m$.
Let $W^\sharp$ denote its Weyl group and $\epsilon\,\colon W^\sharp\to\{\pm1\}$ denote the sign character.
    
The \emph{denominator} of $\fg$ is defined as
    \[
    R = \frac{\prod_{\alpha\in\Phi_0^+}(1-e^{-\alpha})}
    {\prod_{\alpha\in\Phi_1^+}(1+e^{-\alpha})}.
    \]
This is viewed as an element of $\C(\mkern-4mu( e^{-\alpha_1/2},  \ldots,  e^{-\alpha_{2m}/2} )\mkern-4mu)$,  the field of formal Laurent series.
Note that $W^\sharp$ acts on $\C(\mkern-4mu( e^{-\alpha_1/2},  \ldots,  e^{-\alpha_{2m}/2} )\mkern-4mu)$ as a field automorphism by $w\left(\sum_{\lambda\in \frac12 \Z\Delta}c_\lambda e^{-\lambda}\right)=\sum_{\lambda\in \frac12 \Z\Delta}c_\lambda e^{-w(\lambda)}$,  where $w\in W^\sharp$ and $c_\lambda\in\C$.
The next theorem is called the \emph{denominator identity} for $\fg$.

\begin{theorem}[Gorelik \cite{Gorelik12}]\label{finite-denominator}
Let $\rho=\frac12\left(\sum_{\alpha\in\Phi_0^+}\alpha-\sum_{\Phi_1^+}\alpha\right)$ and $S=\{\alpha_{2j-1} \mid j=1,  \ldots,  m\}$. 
Then we have
    \[
    e^\rho R = \sum_{w\in W^\sharp}\epsilon(w)w\left(
    \frac{e^\rho}{\prod_{\beta\in S}(1+e^{-\beta})}\right).
    \]
\end{theorem}
As remarked in Section 3 of Kac--Wakimoto \cite{KacWakimoto1994},  we can rewrite \cref{finite-denominator} as
    \begin{equation}\label{denominator}
    e^\rho R = \frac{1}{2^mm!}\sum_{w\in W}\epsilon(w)w\left(
    \frac{e^\rho}{\prod_{\beta\in S}(1+e^{-\beta})}\right).    
    \end{equation}

\section{Affine Lie superalgebra $\widehat{\spo}(2m,  2m+1)$}\label{Affine Lie superalgebra}

Keep the notation of the previous section,  in particular $\fg=\spo(2m,  2m+1)$.
We define the \emph{affine Lie superalgebra} $\widehat{\fg}=\widehat{\spo}(2m,  2m+1)$ as follows.
As a vector space,  $\widehat{\fg}=\fg\otimes\C[t^{\pm1}]\oplus\C K\oplus\C d$  and the grading is given by $\widehat{\fg}_0=\fg_0\otimes\C[t^{\pm1}]\oplus\C K\oplus\C d$ and $\widehat{\fg}_1=\fg_1\otimes\C[t^{\pm1}]$.
For $X\in\fg$ and $i\in\Z$,  we simply write $X\otimes t^i$ as $Xt^i$.
The Lie bracket is determined by the following conditions:
\begin{itemize}
\item $[Xt^i,  Yt^j]=[X,Y]t^{i+j}+i\delta_{i,-j}(X,Y)K$ for $X,Y\in\fg$ and $i,  j\in\Z$.
Here,  $\delta_{k,l}$ denotes the Kronecker delta.
\item $[\widehat{\fg},  K]=0$.
\item $[d,  Xt^i]=iXt^i$ for $X\in\fg$ and $i\in \Z$.
\end{itemize}

Let $(\cdot,\cdot)\,\colon\widehat{\fg}\times\widehat{\fg}\to\C$ be the even super-symmetric pairing determined by the following conditions:
\begin{itemize}
\item $(Xt^i,  Yt^j)=\delta_{i,-j}(X,Y)$ for $X,Y\in\fg$ and $i,  j\in\Z$.
\item $(\fg\otimes\C[t^{\pm1}], K)=(\fg\otimes\C[t^{\pm1}], d)=0$.
\item $(K,K)=(d,d)=0$ and $(K,d)=1$.
\end{itemize}
One can see that $(\cdot,\cdot)$ is invariant and non-degenerate.

Set $\widehat{\fh}=\fh\oplus\C K\oplus\C d$.
This is called the Cartan subalgebra of $\widehat{\fg}$.
A linear form $\lambda\in\fh^\ast$ is viewed as an element of $\widehat{\fh}^\ast$ such that $\lambda(K)=\lambda(d)=0$.
Since the restriction of $(\cdot,\cdot)$ defines a non-degenerate pairing on $\widehat{\fh}$,  it induces a linear isomorphism $\widehat{\fh}^\ast\xrightarrow{\sim}\widehat{\fh}$.
We define the pairing on $\widehat{\fh}^\ast$ as the pull-back of that on $\widehat{\fh}$ by this isomorphism,  which we again write as $(\cdot,  \cdot)$.
Since $(\cdot,\cdot)|_{\fh^\ast\times\fh^\ast}$ coincides with the original pairing on $\fh^\ast$,  the Weyl groups $W$ and $W^\sharp$ naturally act on $\widehat{\fh}^\ast$.

We define $\delta, \gamma\in\widehat{\fh}^\ast$ by
    \[
    \delta(\fg\otimes\C[t^{\pm1}]) = \gamma(\fg\otimes\C[t^{\pm1}]) = 0,  \qquad
    \delta(K) = \gamma(d) = 0,  \qquad \delta(d) = \gamma(K) = 1.
    \]
Note that $\delta$ and $\gamma$ correspond to $K$ and $d$, respectively,  via the isomorphism $\widehat{\fh}^\ast\to\widehat{\fh}$.
Then,  the root system of $\widehat{\fg}$ with respect to $\widehat{\fh}$ is given as $\widehat{\Phi}=\widehat{\Phi}_0\cup\widehat{\Phi}_1$,  where
    \[
    \widehat{\Phi}_0  = \{\alpha+n\delta \mid \alpha\in\Phi_0,  \ n\in\Z\}
    \cup\{n\delta \mid n\in\Z\setminus\{0\}\},   \qquad
    \widehat{\Phi}_1  = \{\alpha+n\delta \mid \alpha\in\Phi_1,  \ n\in\Z\}.
    \]
For each root $\beta\in\widehat{\Phi}$,  let $\widehat{\fg}_\beta=\{X\in\widehat{\fg} \mid [H,  X]=\beta(H)X \text{ for all $H\in\widehat{\fh}$}\}$ be the corresponding root subspace.
It is easy to see that $\widehat{\fg}_{\alpha+n\delta}=\fg_\alpha\otimes t^n$ and $\widehat{\fg}_{n\delta}=\fh\otimes t^n$ for $\alpha\in\Phi$,  $n\in\Z$.
In particular,  we have the root space decomposition $\widehat{\fg}_0=\widehat{\fh}\oplus\bigoplus_{\beta\in\widehat{\Phi}_0}\widehat{\fg}_\beta$ and $\widehat{\fg}_1=\bigoplus_{\beta\in\widehat{\Phi}_1}\widehat{\fg}_\beta$.
Note that $\dim(\widehat{\fg}_{\alpha+n\delta})=1$ and $\dim(\widehat{\fg}_{n\delta})=2m$ for $\alpha\in\Phi$,  $n\in\Z$. 

Set $\theta=\varepsilon_1+\delta_1\in\Phi_1^+$.
This is the highest root of $\Phi$ with respect to $\Delta$.

\begin{definition}
Set $\alpha_0=\delta-\theta$ and $\widehat{\Delta}=\Delta\cup\{\alpha_0\}=\{\alpha_i \mid i=0,  \ldots,  2m\}$.
\end{definition}

Put $\widehat{\Phi}_i^+=(\sum_{\alpha\in\widehat{\Delta}}\Z_{\geq0}\cdot\alpha)\cap\widehat{\Phi}_i$,  ($i=0,  1$) and $\widehat{\Phi}^+=\widehat{\Phi}_0^+\cup\widehat{\Phi}_1^+$.
A simple calculation shows that 
    \[
    \widehat{\Phi}_0^+  = \Phi_0^+\cup 
    \{ \alpha+n\delta \mid \alpha\in\Phi_0\cup\{0\},\ n\in\Z_{>0} \},  \qquad
    \widehat{\Phi}_1^+ = \Phi_1^+\cup 
    \{\alpha+n\delta \mid \alpha\in\Phi_1, \ n\in\Z_{>0} \}.
    \]
In particular,  we have $\widehat{\Phi}_i=\widehat{\Phi}_i^+\cup(-\widehat{\Phi}_i^+)$,  ($i=0,1$).
This  implies that $\widehat{\Delta}$ is a fundamental set of roots of $\widehat{\Phi}$.

\begin{definition}
The denominator $\widehat{R}$ of $\widehat{\fg}$ is defined as 
    \[
    \widehat{R} = \frac{\prod_{\beta\in\widehat{\Phi}_0^+}
    (1-e^{-\beta})^{\dim(\widehat{\fg}_\beta)}}
    {\prod_{\beta\in\widehat{\Phi}_1^+}
    (1+e^{-\beta})^{\dim(\widehat{\fg}_\beta)}}
    = R\prod_{n=1}^\infty\left( (1-e^{-n\delta})^{2m}\cdot
    \frac{\prod_{\alpha\in\Phi_0^+}(1-e^{-\alpha-n\delta})(1-e^{\alpha-n\delta})}
    {\prod_{\alpha\in\Phi_1^+}(1+e^{-\alpha-n\delta})
    (1+e^{\alpha-n\delta})}\right).
    \]
\end{definition}

Let $M^\sharp$ be the coroot lattice of $\Phi_0^\sharp$,  \emph{i.e.} the free $\Z$-module of rank $m$ spanned by $\{2(\alpha_{2i}+\alpha_{2i+1}) \mid i=1,  \ldots,  m\}$,  where $\alpha_{2m+1}\coloneqq0$.
We define the \emph{affine Weyl group} as $\widehat{W}^\sharp=W^\sharp\ltimes M^\sharp$.

\begin{definition}
For $\alpha\in\fh^\ast$,  we define $t_\alpha\,\colon \widehat{\fh}^\ast\to\widehat{\fh}^\ast$ by
    \[
    t_\alpha(\lambda) = \lambda+\lambda(K)\alpha
    -((\lambda,\alpha)+\tfrac12(\alpha,\alpha)\lambda(K))\delta.
    \]
\end{definition}

One can show that $t_\alpha\in\Aut(\widehat{\fh}^\ast)$, $t_\alpha t_\beta=t_{\alpha+\beta}$, and $wt_\alpha w^{-1}=t_{w(\alpha)}$ for $w\in W$ and $\alpha,  \beta\in\fh^\ast$.
Hence we obtain an action of $\widehat{W}^\sharp$ on $\widehat{\fh}^\ast$.
The next theorem is called the denominator identity for $\widehat{\fg}$.
\begin{theorem}[Gorelik \cite{Gorelik11}]
Let $\widehat{\rho}=\frac12\left(\gamma-\sum_{j=1}^m\alpha_{2j-1}\right)$.
Then we have
    \begin{equation}\label{affine-denominator}
    e^{\widehat{\rho}} \widehat{R} = \sum_{\alpha^\sharp\in M^\sharp} t_{\alpha^\sharp}(
    e^{\widehat{\rho}}R) = \sum_{w\in \widehat{W}^\sharp}\epsilon(w)w\left(
    \frac{e^{\widehat{\rho}}}{\prod_{\beta\in S}(1+e^{-\beta})} \right).
    \end{equation}
\end{theorem}

Let $\xi=\sum_{i=1}^m\varepsilon_i$ so that $(\alpha_j,  \xi)=(-1)^j/2$ for all $j=1,  \ldots,  2m$.
Applying $t_\xi$ to both sides of \eqref{affine-denominator} and using \eqref{denominator},  we obtain
    \begin{equation}\label{t_xi}
    t_\xi(e^{\widehat{\rho}}\widehat{R}) 
    = \frac{1}{2^mm!}\sum_{\alpha^\sharp\in M^\sharp}\sum_{w\in W} 
    \epsilon(w)wt_{\xi+\alpha^\sharp}
    \left( \frac{e^{\widehat{\rho}}}{\prod_{\beta\in S}(1+e^{-\beta})} \right).
    \end{equation}

Write $\alpha^\sharp\in M^\sharp$ as $\alpha^\sharp=\sum_{i=1}^mk_i\cdot 2\varepsilon_i=\sum_{i=1}^m2(k_1+\cdots+k_i)(\alpha_{2i}+\alpha_{2i+1})$ with $\vec{k}=(k_1,  \ldots,  k_m)\in\Z^m$.
We summarize some equations which we use to compute the right-hand side of \eqref{t_xi}.
Each of them is a consequence of an easy calculation.
\begin{lemma}
We have the following equations:
\begin{itemize}
\item $t_{\xi+\alpha^\sharp}(\alpha_{2j-1})=\alpha_{2j-1}+(k_j+1/2)\delta$. 
\item $t_{\xi+\alpha^\sharp}(\widehat{\rho})=\widehat{\rho}+(\xi+\alpha^\sharp)/2-(3m/8+(k_1+\cdots+k_m)+(\alpha^\sharp,  \alpha^\sharp)/4)\delta$.
\item $w(\delta)=\delta$ and $w(\gamma)=\gamma$ for $w\in W$.
\item $(\alpha^\sharp,  \alpha^\sharp)=2\sum_{i=1}^m k_i^2$.
\end{itemize}
\end{lemma}

For an index set $J\subset\{1,  \ldots,  m\}$,  define $\Z^m_J$ by
    \[
    \Z^m_J = \left\{\vec{k}=(k_1,  \ldots,  k_m)\in\Z^m \, \middle|\,
    J=\{j \mid k_j<0\} \right\}.
    \]
The right-hand side of \eqref{t_xi} becomes 
    \begin{align*}
    & \frac{q^{\frac38m}}{2^mm!}
    \sum_{J\subset\{1,\ldots,m\}}\sum_{\vec{k}\in\Z^m_J}
    q^{\sum_{j=1}^m(\frac12k_j^2+k_j)} \\
    & \times \sum_{w\in W}\epsilon(w)w\left( 
    e^{\widehat{\rho}+\frac12(\xi+\alpha^\sharp)} 
     \sum_{\vec{r}\in(\Z_{\geq0})^m}
    \prod_{j\in J}\left(q^{-(k_j+\frac12)}e^{\alpha_{2j-1}}
    (-q^{-(k_j+\frac12)}e^{\alpha_{2j-1}})^{r_j} \right)\cdot
    \prod_{j\not\in J} (-q^{k_j+\frac12}e^{-\alpha_{2j-1}})^{r_j} \right) \\
    & = \frac{q^{\frac38m}}{2^mm!}
    \sum_{J\subset\{1,\ldots,m\}}(-1)^{|J|}
    \sum_{\vec{k},  \vec{r}\in\Z^m_J} (-1)^{\sum_{j=1}^m r_j}
    q^{\frac12\sum_{j=1}^m (k_j^2+2k_j+2k_jr_j+r_j)}  
    \sum_{w\in W}\epsilon(w)w\left( 
    e^{\widehat{\rho}+\frac12(\xi+\alpha^\sharp)
    -\sum_{j=1}^m r_j\alpha_{2j-1}} \right),
    \end{align*}
where we set $q=e^{-\delta}$.
On the other hand,  since $t_\xi(\widehat{\rho})=\widehat{\rho}+\xi/2-3m/8\cdot\delta$,  $t_\xi(\delta)=\delta$ and $t_\xi(\beta)=\beta-(\beta,  \xi)\delta$ for $\beta\in\Phi$,  the left-hand side of \eqref{t_xi} becomes 
    \[
    q^{\frac38m}e^{\frac12(\gamma+\xi-\sum_{j=1}^m\alpha_{2j-1})}
    \left( \prod_{n=1}^\infty (1-q^n)^{2m} \right)
    \frac{\prod_{\beta\in\Phi_0^+}
    \left(\prod_{n=0}^\infty(1-q^{n-(\beta,  \xi)}e^{-\beta})\cdot
    \prod_{n=1}^\infty(1-q^{n+(\beta,  \xi)}e^{\beta}) \right)}
    {\prod_{\beta\in\Phi_1^+}
    \left(\prod_{n=0}^\infty(1+q^{n-(\beta,  \xi)}e^{-\beta})\cdot
    \prod_{n=1}^\infty(1+q^{n+(\beta,  \xi)}e^{\beta}) \right)}.
    \]
Hence we obtain 
    \begin{multline*}\label{heart}\tag{$\heartsuit$}{}
    \begin{split}
    e^{\frac12(\xi-\sum_{j=1}^m\alpha_{2j-1})}
    \left( \prod_{n=1}^\infty (1-q^n)^{2m} \right)
    \frac{\prod_{\beta\in\Phi_0^+}
    \left(\prod_{n=0}^\infty(1-q^{n-(\beta,  \xi)}e^{-\beta})\cdot
    \prod_{n=1}^\infty(1-q^{n+(\beta,  \xi)}e^{\beta}) \right)}
    {\prod_{\beta\in\Phi_1^+}
    \left(\prod_{n=0}^\infty(1+q^{n-(\beta,  \xi)}e^{-\beta})\cdot
    \prod_{n=1}^\infty(1+q^{n+(\beta,  \xi)}e^{\beta}) \right)} \\
    = \frac{1}{2^mm!}
    \sum_{J\subset\{1,\ldots,m\}}(-1)^{|J|}
    \sum_{\vec{k},  \vec{r}\in\Z^m_J} (-1)^{\sum_{j=1}^m r_j}
    q^{\frac12\sum_{j=1}^m (k_j^2+2k_j+2k_jr_j+r_j)}  \\
    \times \sum_{w\in W}\epsilon(w)w\left( 
    e^{\frac12(\xi+\alpha^\sharp)
    -\sum_{j=1}^m (r_j+\frac12)\alpha_{2j-1}} \right).
    \end{split}
    \end{multline*}

Let $R_0=\prod_{\beta\in\Phi_0^+}(1-e^{-\beta})$ be the denominator of $\fg_0$.
From now on,  we divide both sides of \eqref{heart} by $R_0$ and take the limit $e^{-\alpha_j}\to1$,  ($j=1,  \ldots,  2m$).

First,  we deal with $\mathrm{RHS}$ of \eqref{heart}.

\begin{lemma}\label{lem:RHS}
We have
    \begin{multline*}
    \lim_{\substack{e^{-\alpha_{j}}\to1 \\ j=1,  \ldots,  2m}}
    R_0^{-1}  \sum_{w\in W} \epsilon(w)w \left( e^
    {\frac12(\xi+\alpha^\sharp)-\sum_{j=1}^m (r_j+\frac12)\alpha_{2j-1}}\right) \\
    \quad  = \frac{(-1)^m}{\left(\prod_{j=1}^m(2j-1)!\right)^2} 
    \left(\prod_{i=1}^m (1+2r_i)(1+r_i+k_i)\right) 
    \left(\prod_{1\leq i<j\leq m} (r_i-r_j)(1+r_i+r_j)\right) \\
    \times  
    \left(\prod_{1\leq i<j\leq m}(r_i+k_i-r_j-k_j)(2+r_i+k_i+r_j+k_j) \right).
    \end{multline*}
\end{lemma}

\begin{proof}
Let $\rho_0=\frac12\sum_{\beta\in\Phi_0^+}\beta$ be the Weyl vector of $\Phi_0^+$.  
Set 
    \[
    \lambda = -\rho_0
    +\tfrac12(\xi+\alpha^\sharp)-\sum_{j=1}^m (r_j+\tfrac12)\alpha_{2j-1}
    = -\rho_0+\sum_{i=1}^m(k_i+\tfrac12)\varepsilon_i
    -\sum_{j=1}^m(r_j+\tfrac12)\alpha_{2j-1}.
    \]
From the Weyl dimension formula (and its proof,  \cite[Corollary 24.3]{Humphreys}),  we obtain
    \[
    \lim_{\substack{e^{-\alpha_{j}}\to1 \\ j=1,  \ldots,  2m}}
    R_0^{-1} \sum_{w\in W} \epsilon(w)w(e^{\rho_0+\lambda}) \quad
    = \prod_{\beta\in\Phi_0^+}\frac{(\rho_0+\lambda,  \beta)}
    {(\rho_0,  \beta)}.
    \]

First,  $(\rho_0+\lambda,  \beta)$ for $\beta\in\Phi_0^+$ is computed as follows:
\begin{itemize}
\item $(\rho_0+\lambda,  \varepsilon_i-\varepsilon_j)=(r_i+k_i-r_j-k_j)/2$ for $1\leq i<j\leq m$.
\item $(\rho_0+\lambda,  \varepsilon_i+\varepsilon_j)=(2+r_i+k_i+r_j+k_j)/2$ for $1\leq i<j\leq m$.
\item $(\rho_0+\lambda,  2\varepsilon_p)=1+r_p+k_p$ for $1\leq p\leq m$.

\item $(\rho_0+\lambda,  \delta_i-\delta_j)=(r_i-r_j)/2$ for $1\leq i<j\leq m$.
\item $(\rho_0+\lambda,  \delta_i+\delta_j)=(1+r_i+r_j)/2$ for $1\leq i<j\leq m$.
\item $(\rho_0+\lambda,  \delta_q)=(1+2r_q)/4$ for $1\leq q\leq m$.
\end{itemize}
Hence we get
    \begin{align*}
    \prod_{\beta\in\Phi_0^+}(\rho_0+\lambda,  \beta)
    =2^{-|\Phi_0^+|} \left(\prod_{i=1}^m (1+2r_i)(1+r_i+k_i)\right) 
    \left(\prod_{1\leq i<j\leq m} (r_i-r_j)(1+r_i+r_j)\right) \\
    \times  
    \left(\prod_{1\leq i<j\leq m}(r_i+k_i-r_j-k_j)(2+r_i+k_i+r_j+k_j) \right).
    \end{align*}

On the other hand,  since we have $\rho_0=\sum_{p=1}^m (m+1-p)\varepsilon_p+\sum_{q=1}^m(m+1/2-q)\delta_q$,  the pairing $(\rho_0,  \beta)$ for $\beta\in\Phi_0^+$ is computed as 
\begin{itemize}
\item $(\rho_0,  \varepsilon_i-\varepsilon_j)=(j-i)/2$ and 
    \[
    \prod_{1\leq i<j\leq m}(\rho_0,  \varepsilon_i-\varepsilon_j)
    = 2^{-\frac{m(m-1)}{2}}\prod_{1\leq i<m}(m-i)! = 2^{-\frac{m(m-1)}{2}}\prod_{i=1}^{m-1}i!.
    \]
\item $(\rho_0,  \varepsilon_i+\varepsilon_j)=(2m+2-i-j)/2$ and
    \[
    \prod_{1\leq i<j\leq m}(\rho_0,  \varepsilon_i+\varepsilon_j)
    = 2^{-\frac{m(m-1)}{2}}\prod_{1\leq i<m}\frac{(2m-2i+1)!}{(m-i+1)!}
    = 2^{-\frac{m(m-1)}{2}}\prod_{i=1}^{m-1}\frac{(2i+1)!}{(i+1)!}.
    \]
\item $(\rho_0,  2\varepsilon_p)=m+1-p$ and $\prod_{p=1}^m(\rho_0,  2\varepsilon_p)=m!$.

\item $(\rho_0,  \delta_i-\delta_j)=(i-j)/2$ and $\prod_{1\leq i<j\leq m}(\rho_0,  \delta_i-\delta_j)=(-2)^{-m(m-1)/2}\prod_{i=1}^{m-1}i!$.
\item $(\rho_0,  \delta_i+\delta_j)=(i+j-2m-1)/2$ and
    \[
    \prod_{1\leq i<j\leq m}(\rho_0,  \delta_i+\delta_j) 
    = (-2)^{-\frac{m(m-1)}{2}}\prod_{1\leq i<m}\frac{(2m-2i)!}{(m-i)!}
    = (-2)^{-\frac{m(m-1)}{2}}\prod_{i=1}^{m-1}\frac{(2i)!}{i!}.
    \]
\item $(\rho_0,  \delta_q)=(2q-2m-1)/4$ and
    \[
    \prod_{q=1}^m(\rho_0,  \delta_q) = (-4)^{-m}\prod_{i=1}^m(2i-1) 
    = (-4)^{-m}(2m)!\cdot \prod_{i=1}^m(2i)^{-1}.
    \]
\end{itemize}
Hence we have $\prod_{\beta\in\Phi_0^+}(\rho_0,  \beta)=(-1)^m2^{-2m^2}\left(\prod_{i=1}^m(2j-1)!\right)^2=(-1)^m2^{-|\Phi_0^+|}\left(\prod_{j=1}^m(2j-1)!\right)^2$.
This completes the proof.
\end{proof}

\begin{corollary}
We have
    \begin{align}\label{eq:RHS}
    \begin{split}
    & \lim_{\substack{e^{-\alpha_{j}}\to1 \\ j=1,  \ldots,  2m}} 
    \frac{\mathrm{RHS}\text{ of }\eqref{heart}}{R_0}  \\  
    & \quad = \frac{(-1)^m}{m!\left(\prod_{j=1}^m(2j-1)!\right)^2} 
    \sum_{\vec{k},  \vec{r}\in(\Z_{\geq0})^m} (-1)^{\sum_{j=1}^m r_j}
    q^{\frac12\sum_{j=1}^m (k_j^2+2k_j+2k_jr_j+r_j)}  
    \left(\prod_{i=1}^m (1+2r_i)(1+r_i+k_i)\right) \\ 
    & \hspace{40pt} \times  
    \left(\prod_{1\leq i<j\leq m}(r_i-r_j)(1+r_i+r_j)(r_i+k_i-r_j-k_j)(2+r_i+k_i+r_j+k_j) \right).
    \end{split}
    \end{align}
\end{corollary}

\begin{proof}
From \cref{lem:RHS},  we obtain
    \begin{align*}
    \lim_{\substack{e^{-\alpha_{j}}\to1 \\ j=1,  \ldots,  2m}} 
    \frac{\mathrm{RHS}\text{ of }\eqref{heart}}{R_0} 
    & \ = \frac{(-1)^m}{2^mm!\left(\prod_{j=1}^m(2j-1)!\right)^2}  
    \sum_{J\subset\{1,\ldots,m\}}(-1)^{|J|} \\ 
    & \quad \times \sum_{\vec{k},  \vec{r}\in\Z^m_J} (-1)^{\sum_{j=1}^m r_j}
    q^{\frac12\sum_{j=1}^m (k_j^2+2k_j+2k_jr_j+r_j)}  
    \left(\prod_{i=1}^m (1+2r_i)(1+r_i+k_i)\right) \\ 
    & \quad \times  
    \left(\prod_{1\leq i<j\leq m}(r_i-r_j)(1+r_i+r_j)(r_i+k_i-r_j-k_j)(2+r_i+k_i+r_j+k_j) \right).
    \end{align*}
Applying the change of variables $k_j\mapsto-k_j-1$ and $r_j\mapsto-r_j-1$ for $j\in J$,  the summand of the above equation is multiplied by $(-1)^{|J|}$ and the summation becomes over $\vec{k},  \vec{r}\in(\Z_{\geq0})^m$.
This concludes the proof.
\end{proof}

Next,  we deal with $\mathrm{LHS}$ of \eqref{heart}.
\begin{lemma}
    \begin{align}\label{eq:LHS}
    \lim_{\substack{e^{-\alpha_{j}}\to1 \\ j=1,  \ldots,  2m}} 
    \frac{\mathrm{LHS}\text{ of }\eqref{heart}}{R_0} 
    & = (-1)^{\frac{m(m+1)}{2}} q^{\frac12\cdot \frac{m(m-1)}{2}} \cdot 
    \left( \prod_{n=1}^\infty \frac{1-q^{\frac12\cdot 2n}}{1+q^{\frac12(2n-1)}} \right)^{2m(2m+1)}.
    \end{align}
\end{lemma}
\begin{proof}
Let $\beta\in\Phi^+$.
It is easy to see that 
    \[
    (\beta,  \xi) = 
        \begin{cases}
        1 & \text{$\beta\in\{\varepsilon_i+\varepsilon_j, \ (1\leq i<j\leq m),  \
         2\varepsilon_p,  \ (1\leq p\leq m)\}$},  \\
        \frac12 & \text{$\beta\in\{\varepsilon_i-\delta_j,  \ (1\leq i<j\leq m),  \ 
        \varepsilon_i+\delta_j, \ (1\leq i,  j\leq m),  \ 
        \varepsilon_p,  \ (1\leq p\leq m)\}$},  \\
        -\frac12 & \text{$\beta\in\{\delta_i-\varepsilon_j, \ (1\leq i\leq j\leq m)\}$},  \\
        0 & \text{otherwise}.
        \end{cases}
    \]
If $(\beta,  \xi)=0$,  then 
    \begin{align*}
    (1-e^{-\beta})^{-1}
    \prod_{n=0}^\infty(1-q^{n-(\beta,  \xi)}e^{-\beta})
    \prod_{n=1}^\infty(1-q^{n+(\beta,  \xi)}e^{\beta}) 
    & = \prod_{n=1}^\infty (1-q^ne^{-\beta})(1-q^ne^\beta) \\
    & \quad \underset{j=1,  \ldots,  2m}{\xrightarrow{e^{-\alpha_{j}}\to1}} \quad
    \left(\prod_{n=1}^\infty (1-q^n)\right)^2.
    \end{align*}
If $(\beta,  \xi)=1$,  then 
    \begin{multline*}
    (1-e^{-\beta})^{-1}
    \prod_{n=0}^\infty(1-q^{n-(\beta,  \xi)}e^{-\beta})
    \prod_{n=1}^\infty(1-q^{n+(\beta,  \xi)}e^{\beta}) 
    = (1-q^{-1}e^{-\beta}) \prod_{n=1}^\infty 
    (1-q^ne^{-\beta})(1-q^{n+1}e^{\beta}) \\
     = -q^{-1}e^{-\beta}
    \prod_{n=1}^\infty (1-q^ne^{-\beta})(1-q^ne^{\beta}) \
    \underset{j=1,  \ldots,  2m}{\xrightarrow{e^{-\alpha_{j}}\to1}} \
    -q^{-1} \left(\prod_{n=1}^\infty(1-q^n)\right)^2.
    \end{multline*}
If $(\beta,  \xi)=-1/2$,  then 
    \begin{multline*}
    \prod_{n=0}^\infty(1+q^{n-(\beta,  \xi)}e^{-\beta})
    \prod_{n=1}^\infty(1+q^{n+(\beta,  \xi)}e^{\beta})
    = \prod_{n=1}^\infty (1+q^{n-\frac12}e^{-\beta})
    (1+q^{n-\frac12}e^{\beta}) \\
    \underset{j=1,  \ldots,  2m}{\xrightarrow{e^{-\alpha_{j}}\to1}} \
     \left(\prod_{n=1}^\infty(1+q^{\frac12(2n-1)})\right)^2.
    \end{multline*}
If $(\beta,  \xi)=1/2$,  then 
    \begin{multline*}
    \prod_{n=0}^\infty(1+q^{n-(\beta,  \xi)}e^{-\beta})
    \prod_{n=1}^\infty(1+q^{n+(\beta,  \xi)}e^{\beta})
    = (1+q^{-\frac12}e^{-\beta})\prod_{n=1}^\infty (1+q^{n-\frac12}e^{-\beta})
    (1+q^{n+\frac12}e^{\beta}) \\
    = q^{-\frac12}e^{-\beta}\cdot 
    \prod_{n=1}^\infty(1+q^{\frac12(2n-1)}e^{-\beta})
    (1+q^{\frac12(2n-1)}e^{\beta}) \
    \underset{j=1,  \ldots,  2m}{\xrightarrow{e^{-\alpha_{j}}\to1}} \
     q^{-\frac12}\left(\prod_{n=1}^\infty(1+q^{\frac12(2n-1)})\right)^2.
    \end{multline*}

Note that $\#\{\beta\in\Phi_0 \mid (\beta,  \xi)=1\}=m(m+1)/2$ and $\#\{\beta\in\Phi_0 \mid (\beta,  \xi)=1/2\}=m(3m+1)/2$.
Thus we obtain
    \begin{align}
    \begin{split}
    \lim_{\substack{e^{-\alpha_{j}}\to1 \\ j=1,  \ldots,  2m}} 
    \frac{\mathrm{LHS}\text{ of }\eqref{heart}}{R_0} 
    & = \left( \prod_{n=1}^\infty \frac{1-q^n}{1+q^{\frac12(2n-1)}} \right)^{2m(2m+1)}  
    \cdot (-q^{-1})^{\frac{m(m+1)}{2}} \cdot 
    (q^{-\frac12})^{-\frac{m(3m+1)}{2}} \\ 
    & = (-1)^{\frac{m(m+1)}{2}} q^{\frac{m(m-1)}{4}} \cdot 
    \left( \prod_{n=1}^\infty \frac{1-q^n}{1+q^{\frac12(2n-1)}} \right)^{2m(2m+1)}.
    \end{split}
    \end{align}
\end{proof}

Replace $q^{1/2}$ by $-q$ in \eqref{eq:RHS} and \eqref{eq:LHS}.
Using \eqref{eq:triangle},  we obtain \cref{thm:Kac-Wakimoto-Conj}.

\section{Modular framework for Kac--Wakimoto's identity}\label{Modular-framework}

It is a classical fact that multiplying the generating function $\triangle(q)$ of triangular numbers by $q^{1/8}$ yields a modular form of weight $1/2$. This emphasizes that the right-hand side of \cref{thm:Kac-Wakimoto-Conj} is also a modular form. In this section, we introduce Vign\'{e}ras' theorem, which is essential for providing a framework to directly establish this modularity. The next section applies this framework to offer an alternative proof of \cref{thm:Kac-Wakimoto-Conj} as an identity between modular forms. To clarify this approach, we begin by briefly reviewing some known properties of the generating function $\triangle(q)$ on the left-hand side. 

\subsection{The generating function $\theta_\triangle(\tau)$}\label{generating-function}

Let $\bbH \coloneqq \{\tau \in \C \mid \Im(\tau) > 0\}$ denote the upper half-plane and set $q = e^{2\pi i\tau}$. For convenience in the subsequent discussion, we consider the transformed version
\[
	\theta_\triangle(\tau) \coloneqq q^{1/16} \triangle(q^{1/2}) = \frac{1}{2} \sum_{n \in \frac{1}{2} + \Z} q^{\frac{n^2}{4}},
\]
obtained by setting $q \mapsto q^{1/2}$. It is known that using the Dedekind eta function $\eta(\tau) = q^{1/24} \prod_{n=1}^\infty (1-q^n)$, it can be expressed as
\begin{align}\label{eq:theta-eta}
	\theta_\triangle(\tau) = \frac{\eta(\tau)^2}{\eta(\tau/2)},
\end{align}
which can be found in K\"{o}hler's monograph~\cite[(8.5)]{Kohler2011}. Under this notation, the following holds.

\begin{lemma}\label{lem:theta-triangle}
	The function $\theta_\triangle(\tau)$ is a modular form of weight $1/2$ on $\Gamma(2)$ with no zeros or poles on $\bbH$. Specifically, for generators $\smat{1 & 2 \\ 0 & 1}$ and $\smat{1 & 0 \\ 2 & 1}$ of $\Gamma(2)/\{\pm I\}$, we have
	\begin{align*}
		\theta_\triangle(\tau+2) &= e^{\frac{\pi i}{4}} \theta_\triangle(\tau),\\
		\theta_\triangle \left(\frac{\tau}{2\tau+1}\right) &= (2\tau+1)^{1/2} \theta_\triangle(\tau).
	\end{align*}
	Its behavior at the cusps $i\infty, 0, 1$ of $\Gamma(2)$ is given by
	\begin{align*}
		\theta_\triangle(\tau) &= q^{1/16} + O(q^{9/16}),\\
		(-i\tau)^{-1/2} \theta_\triangle\left(-\frac{1}{\tau}\right) &= \frac{1}{\sqrt{2}} + O(q),\\
		(-i\tau)^{-1/2} \theta_\triangle \left(\frac{\tau-1}{\tau}\right) &= e^{\frac{\pi i}{8}} q^{1/16} + O(q^{9/16}),
	\end{align*}
	respectively.
\end{lemma}

\begin{proof}
The expression as an eta quotient in~\eqref{eq:theta-eta} implies that it has no zeros or poles on $\bbH$. The transformation laws for the action of $\Gamma(2)$ and the Fourier series expansions at the cusps follow from the well-known transformation laws
\[
	\eta(\tau+1) = e^{\frac{\pi i}{12}} \eta(\tau), \quad \eta \left(-\frac{1}{\tau}\right) = (-i \tau)^{1/2} \eta(\tau),
\]
and the identity
\[
	\eta(\tau+1/2) = e^{\frac{\pi i}{24}} \frac{\eta(2\tau)^3}{\eta(\tau) \eta(4\tau)},
\]
(see K\"{o}hler~\cite[(1.13), (1.14), and Proposition 1.5]{Kohler2011}).
\end{proof}

On the other hand, the right-hand side of \cref{thm:Kac-Wakimoto-Conj} exhibits two notable differences from classical theta functions. First, the sum is restricted to a cone region, rather than spanning the entire lattice $\Z^{2m}$. Second, the quadratic form $\sum_{j=1}^m k_j^2 + 2k_j r_j$ in the exponent of $q$ is indefinite, rather than positive definite. However, the first difference is superficial and can be resolved by rewriting the expressions as follows. To facilitate this reformulation, we introduce the following polynomial.

\begin{definition}\label{def:Vm}
	For each $m \in \Z_{>0}$, we define
	\[
		V_m(\vec{\bm{x}}) \coloneqq \prod_{1 \le i < j \le m} (x_i^2 - x_j^2) (y_i^2 - y_j^2) \times \prod_{j=1}^m x_j y_j,
	\]
	where we set $\vec{\bm{x}} = (\bm{x}_1, \dots, \bm{x}_m)$ and $\bm{x}_j = \smat{x_j \\ y_j}$ for later convenience.
\end{definition}

\begin{theorem}\label{thm:KW-rewrite}
	For each $m \ge 1$, we define
	\begin{align*}
		\KW_m(\tau) &\coloneqq \frac{1}{4^m} \frac{1}{m! \left(\prod_{j=1}^m (2j-1)!\right)^2} \sum_{\substack{x_1, \dots, x_m \in \Z \\ y_1, \dots, y_m \in \frac{1}{2} + \Z}} \prod_{j=1}^m \bigg(\sgn(x_j-y_j) - \sgn(-x_j-y_j) \bigg) \\
			&\qquad \times V_m(\vec{\bm{x}}) \prod_{j=1}^m q^{\frac{1}{2} (x_j^2 - y_j^2)} (-1)^{x_j-y_j-1/2}.
	\end{align*}
	Then \cref{thm:Kac-Wakimoto-Conj} is equivalent to the identity
	\begin{align}\label{eq:KW-identity}
		\theta_\triangle(\tau)^{2m(2m+1)} = \mathrm{KW}_m(\tau).
	\end{align}
\end{theorem}

\begin{proof}
By changing variables via $k_j = x_j - y_j - 1/2, r_j = y_j-1/2$, and $q \mapsto q^{1/2}$ in \cref{thm:Kac-Wakimoto-Conj}, we obtain
\begin{align}\label{eq:KW-cone}
	\theta_\triangle(\tau)^{2m(2m+1)} &= \frac{2^m}{m! \left(\prod_{j=1}^m (2j-1)!\right)^2} \sum_{\substack{x_1, \dots, x_m \in \Z_{>0} \\ y_1, \dots, y_m \in \frac{1}{2} + \Z_{\ge 0} \\ x_j > y_j}} V_m(\vec{\bm{x}}) \prod_{j=1}^m q^{\frac{1}{2}(x_j^2 - y_j^2)} (-1)^{(x_j-y_j-1/2)}.
\end{align}
By focusing on the symmetry of the transformation $(x_j, y_j) \mapsto (\pm x_j, \pm y_j)$ with all four possible sign changes of the summand, the right-hand side can be rewritten as
\begin{align*}
	= \frac{1}{2^m} \frac{1}{m! \left(\prod_{j=1}^m (2j-1)!\right)^2} \sum_{\substack{x_1, \dots, x_m \in \Z \\ y_1, \dots, y_m \in \frac{1}{2} + \Z \\ x_j^2 - y_j^2 > 0}} V_m(\vec{\bm{x}}) \prod_{j=1}^m \sgn(x_j) q^{\frac{1}{2}(x_j^2 - y_j^2)} (-1)^{(x_j-y_j-1/2)},
\end{align*}
which immediately leads to the desired result.
\end{proof}

\subsection{Vign\'{e}ras' criterion for constructing non-holomorphic modular forms}\label{Vigneras}

Our goal is to construct a framework for theta functions associated with indefinite quadratic forms that enables the series $\KW_m(\tau)$ to be interpreted as a modular form. This will be achieved by applying Vign\'{e}ras' criterion for theta functions associated with general quadratic forms. For a positive integer $n$ and a pair of non-negative integers $(r,s)$ with $n = r+s$, we consider the quadratic form
\begin{align*}
	Q(\bm{x}) = \frac{1}{2} {}^t \bm{x} A \bm{x} = \frac{1}{2} \left(\sum_{j=1}^r x_j^2 - \sum_{j=r+1}^n x_j^2 \right),
\end{align*}
of signature $(r,s)$, where $A = \mathrm{diag}(\underbrace{1, \dots, 1}_r, \underbrace{-1, \dots, -1}_s)$. The associated bilinear form is given by $B(\bm{x}, \bm{y}) = {}^t \bm{x} A \bm{y}$. In terms of the Euler operator $\calE$ and the Laplace operator $\Delta$, defined by
\begin{align*}
	\mathcal{E} \coloneqq \sum_{j=1}^n x_j \frac{\partial}{\partial x_j}, \qquad \Delta \coloneqq \sum_{j=1}^r \frac{\partial^2}{\partial x_j^2} - \sum_{j=r+1}^n \frac{\partial^2}{\partial x_j^2},
\end{align*}
we introduce Vign\'{e}ras' operator $\calD$ as
\[
	\calD \coloneqq \mathcal{E} - \frac{\Delta}{4\pi}.
\]

\begin{theorem}[Vign\'{e}ras~\cite{Vigneras1977-proc}]\label{fact:Vig}
	Let $p\,\colon \R^n \to \C$ be a function satisfying the following assumptions:
	\begin{itemize}
		\item[$(1)$] $p(\bm{x}) e^{-2\pi Q(\bm{x})}$ is a Schwartz function.
		\item[$(2)$] $p(\bm{x})$ is an eigenfunction of the operator $\calD$ with the eigenvalue $\lambda \in \Z$, that is, $\calD p(\bm{x}) = \lambda p(\bm{x})$.
	\end{itemize}
	Then, for $\bm{a}, \bm{b} \in \R^n$, the theta function defined by
	\[
		\theta_{\bm{a}, \bm{b}}(\tau) \coloneqq v^{-\lambda/2} \sum_{\bm{x} \in \bm{a} + \Z^n} p (\bm{x} \sqrt{v}) q^{Q(\bm{x})} e^{2\pi i B(\bm{x}, \bm{b})}, \qquad (\tau = u + iv)
	\]
	converges absolutely and satisfies
	\[
		\theta_{\bm{a}, \bm{b}} \left(-\frac{1}{\tau}\right) = i^{-s-\lambda} (-i\tau)^{\lambda+n/2} e^{2\pi i B(\bm{a}, \bm{b})} \theta_{-\bm{b}, \bm{a}} (\tau).
	\]
\end{theorem}

\section{Modularity of the Kac--Wakimoto series}\label{Modularity-Kac--Wakimoto}

To capture the series $\KW_m(\tau)$ using Vign\'{e}ras' criterion, our task is to find an appropriate function $p(\bm{x})$ in the case $r = s = m$. Once the modular transformation laws satisfied by $\KW_m(\tau)$ are obtained, we can compare them with \cref{lem:theta-triangle} to derive the desired identity in \cref{thm:KW-rewrite}.

\subsection{Indefinite theta functions of signature $(m,m)$}\label{indefinite-theta}

For the case $s=1$, Zwegers~\cite{Zwegers2002} developed a method to construct $p(\bm{x})$, which played a key role in understanding Ramanujan's mock theta functions. This approach was later generalized by Nazaroglu~\cite{Nazaroglu2018} to treat general signatures and by Roehrig--Zwegers~\cite{RoehrigZwegers2022, RoehrigZwegers2022-pre} to include spherical polynomials when $s=1$. Building on these advancements, we provide a further extension of their method, applicable to $\KW_m(\tau)$ for the signature $(m,m)$ with spherical polynomials.

To simplify notation, we represent elements $(x_1, \dots, x_m, y_1, \dots, y_m) \in \R^{2m}$ as $\vec{\bm{x}} = (\bm{x}_1, \dots, \bm{x}_m)$, where $\bm{x}_j = \smat{x_j \\ y_j}$ for $1 \le j \le m$. For a vector $\vec{x} = (x_1, \dots, x_m) \in \R^m$, we set $|\vec{x}| \coloneqq x_1 + \cdots + x_m$. To clarify the dimension, we attach the subscript $m$ to symbols such as $Q, B, \Delta$, and $\calD$. For instance,
\[
	Q_m(\vec{\bm{x}}) = \frac{1}{2} \sum_{j=1}^m (x_j^2 - y_j^2).
\]
Additionally, defining 
\[
	\calD_1^{(j)} = x_j \frac{\partial}{\partial x_j} + y_j \frac{\partial}{\partial y_j} - \frac{1}{4\pi} \left(\frac{\partial^2}{\partial x_j^2} - \frac{\partial^2}{\partial y_j^2} \right),
\]
the operator $\calD_m$ can be decomposed as $\calD_m = \sum_{j=1}^m \calD_1^{(j)}$. In the case $m =1$, the region $\{\bm{x} \in \R^2 \mid Q_1(\bm{x}) < 0\}$ consists of two components. We fix one component, for instance,
\[
	\calC \coloneqq \{\bm{x} \in \R^2 \mid Q_1(\bm{x}) < 0, y > 0\}.
\]
The error function, defined by
\[
	E(z) \coloneqq 2\int_0^z e^{-\pi u^2} \dd u,
\]
has been widely used since Zwegers~\cite{Zwegers2002} to smoothly approximate the sign function. Under these settings, we introduce candidate functions that satisfy Vign\'{e}ras' criterion, extending the function constructed by Roehrig--Zwegers in the case $m=1$.

\begin{definition}\label{def:deg-m-p}
	For $\vec{\bm{c}}_0 = (\bm{c}_1^{(0)}, \dots, \bm{c}_m^{(0)}), \vec{\bm{c}}_1 = (\bm{c}_1^{(1)}, \dots, \bm{c}_m^{(1)}) \in \mathcal{C}^m$ and a polynomial $f(\vec{\bm{x}})$ \bf of degree $d$,  we define
	\begin{align*}
		p^{\vec{\bm{c}}_0, \vec{\bm{c}}_1}[f](\vec{\bm{x}}) \coloneqq \sum_{\vec{\sigma} = (\sigma_1, \dots, \sigma_m) \in \{0,1\}^m} (-1)^{|\vec{\sigma}|} \sum_{k=0}^d \frac{(-1)^k}{(4\pi)^k} \sum_{\substack{\vec{k} = (k_1, \dots, k_m) \in (\Z_{\ge 0})^m \\ |\vec{k}| = k}} F_{\vec{\sigma}, \vec{k}}(\vec{\bm{x}}) G_{\vec{\sigma}, \vec{k}}(\vec{\bm{x}}),
	\end{align*}
	where we put
	\begin{align*}
		F_{\vec{\sigma}, \vec{k}}(\vec{\bm{x}}) &\coloneqq \prod_{j=1}^m \frac{1}{k_j!} E^{(k_j)} \left(\frac{B_1(\bm{c}_j^{(\sigma_j)}, \bm{x}_j)}{\sqrt{-Q_1(\bm{c}_j^{(\sigma_j)})}} \right),\\
		G_{\vec{\sigma}, \vec{k}}(\vec{\bm{x}}) &\coloneqq \left(\prod_{j=1}^m \partial_{\bm{c}_j^{(\sigma_j)}}(\bm{x}_j)^{k_j}\right) f(\vec{\bm{x}}),\\
		\partial_{\bm{c}}(\bm{x}) &\coloneqq \frac{1}{\sqrt{-Q_1(\bm{c})}} \left(c_1 \frac{\partial}{\partial x} + c_2 \frac{\partial}{\partial y}\right), \qquad \bm{c} = \pmat{c_1 \\ c_2}, \bm{x} = \pmat{x \\ y}.
	\end{align*}
\end{definition}

We will now verify that this construction satisfies the conditions of the criterion.

\begin{proposition}\label{prop:degm-Vig1}
	The function $p^{\vec{\bm{c}}_0, \vec{\bm{c}}_1}[f](\vec{\bm{x}}) e^{-2\pi Q_m(\vec{\bm{x}})}$ is a Schwartz function.
\end{proposition}

\begin{proof}
It suffices to show that the function
\begin{align}\label{eq:def-Sk}
	S_{\vec{k}}(\vec{\bm{x}}) \coloneqq \sum_{\vec{\sigma} \in \{0,1\}^m} (-1)^{|\vec{\sigma}|} F_{\vec{\sigma}, \vec{k}}(\vec{\bm{x}}) G_{\vec{\sigma}, \vec{k}}(\vec{\bm{x}}) e^{-2\pi Q_m(\vec{\bm{x}})}
\end{align}
is a Schwartz function for each $\vec{k} = (k_1, \dots, k_m) \in (\Z_{\ge 0})^m$. Let $I(\vec{k}) \coloneqq \{1 \le j \le m \mid k_j = 0\}$ and its complement $I(\vec{k})^c = \{1, 2, \dots, m\} \setminus I(\vec{k})$. It is straightforward to verify that
\begin{align}\label{eq:Sk-explicit}
\begin{split}
	S_{\vec{k}}(\vec{\bm{x}}) &= \sum_{\vec{\sigma}' \in \{0,1\}^{I(\vec{k})^c}} (-1)^{|\vec{\sigma}'|} \prod_{j \in I(\vec{k})} \left(E\left(\frac{B_1(\bm{c}_j^{(0)}, \bm{x}_j)}{\sqrt{-Q_1(\bm{c}_j^{(0)})}} \right) - E\left(\frac{B_1(\bm{c}_j^{(1)}, \bm{x}_j)}{\sqrt{-Q_1(\bm{c}_j^{(1)})}} \right)\right) e^{-2\pi Q_1(\bm{x}_j)}\\
		&\qquad \times \prod_{j \in I(\vec{k})^c} \frac{1}{k_j!} E^{(k_j)} \left(\frac{B_1(\bm{c}_j^{(\sigma'_j)}, \bm{x}_j)}{\sqrt{-Q_1(\bm{c}_j^{(\sigma'_j)})}} \right) e^{-2\pi Q_1(\bm{x}_j)} \times \left(\prod_{j \in I(\vec{k})^c} \partial_{\bm{c}_j}^{(\sigma'_j)} (\bm{x}_j)^{k_j}\right) f(\vec{\bm{x}}).
\end{split}
\end{align}
The fact that this is a Schwartz function follows essentially from the proof of~\cite[Lemma 3.1]{RoehrigZwegers2022}. Indeed, for $k>0$ and $\bm{c} \in \calC$, there exists a polynomial $P(\bm{x})$ such that 
\[
	E^{(k)} \left(\frac{B_1(\bm{c}, \bm{x})}{\sqrt{-Q_1(\bm{c})}}\right) e^{-2\pi Q_1(\bm{x})} = P(\bm{x}) e^{-2\pi \left(Q_1(\bm{x}) -\frac{1}{2} \frac{B_1(\bm{c}, \bm{x})^2}{Q_1(\bm{c})} \right)}.
\]
Since $Q_1(\bm{x}) - B_1(\bm{c}, \bm{x})^2/2Q_1(\bm{c})$ is positive definite, as shown in~\cite[Lemma 2.5]{Zwegers2002}, we see that $S_{\vec{k}}(\vec{\bm{x}})$ is a Schwartz function in $\bm{x}_j$ for $j \in I(\vec{k})^c$. As for $j \in I(\vec{k})$, it is known from the proof of~\cite[Lemma 3.1]{RoehrigZwegers2022} that
\[
	\left(E\left(\frac{B_1(\bm{c}_0, \bm{x})}{\sqrt{-Q_1(\bm{c}_0)}} \right) - E\left(\frac{B_1(\bm{c}_1, \bm{x})}{\sqrt{-Q_1(\bm{c}_1)}} \right) \right) e^{-2\pi Q_1(\bm{x})}
\]
is a Schwartz function for $\bm{c}_0, \bm{c}_1 \in \calC$. This implies that $S_{\vec{k}}(\vec{\bm{x}})$ is a Schwartz function in $\bm{x}_j$ for $j \in I(\vec{k})$ as well.
\end{proof}

\begin{definition}
	A polynomial $f(\vec{\bm{x}})$ is called a \emph{spherical polynomial} of degree $d$ if it is homogeneous of degree $d$ and annihilated by the Laplace operator $\Delta_m$.
\end{definition}

\begin{proposition}\label{prop:degm-Vig2}
	For a spherical polynomial $f(\vec{\bm{x}})$ of degree $d$, we have $\mathcal{D}_m p^{\vec{\bm{c}}_0, \vec{\bm{c}}_1}[f](\vec{\bm{x}}) = d p^{\vec{\bm{c}}_0, \vec{\bm{c}}_1}[f](\vec{\bm{x}})$.
\end{proposition}

\begin{proof}
For each $\vec{k}$ with $|\vec{k}| = k$, a direct calculation yields
\begin{align*}
	\calD_m \left(F_{\vec{\sigma}, \vec{k}}(\vec{\bm{x}}) G_{\vec{\sigma}, \vec{k}}(\vec{\bm{x}}) \right) &= \calD_m F_{\vec{\sigma}, \vec{k}}(\vec{\bm{x}}) \cdot G_{\vec{\sigma}, \vec{k}}(\vec{\bm{x}}) + F_{\vec{\sigma}, \vec{k}}(\vec{\bm{x}}) \cdot \calD_m G_{\vec{\sigma}, \vec{k}}(\vec{\bm{x}})\\
		&\quad -\frac{1}{2\pi} \sum_{j=1}^m \left( \frac{\partial}{\partial x_j} F_{\vec{\sigma}, \vec{k}}(\vec{\bm{x}}) \cdot \frac{\partial}{\partial x_j} G_{\vec{\sigma}, \vec{k}}(\vec{\bm{x}}) - \frac{\partial}{\partial y_j} F_{\vec{\sigma}, \vec{k}}(\vec{\bm{x}}) \cdot \frac{\partial}{\partial y_j} G_{\vec{\sigma}, \vec{k}}(\vec{\bm{x}}) \right).
\end{align*}
Since
\[
	\calD_1 E^{(k)} \left(\frac{B_1(\bm{c}, \bm{x})}{\sqrt{-Q_1(\bm{c})}} \right) = -k E^{(k)} \left(\frac{B_1(\bm{c}, \bm{x})}{\sqrt{-Q_1(\bm{c})}} \right)
\]
is known from the proof of~\cite[Lemma 3.2]{RoehrigZwegers2022}, for the first term, we have
\begin{align*}
	\calD_m F_{\vec{\sigma}, \vec{k}}(\vec{\bm{x}}) &= \sum_{j=1}^m \calD_1^{(j)} F_{\vec{\sigma}, \vec{k}}(\vec{\bm{x}}) = - k F_{\vec{\sigma}, \vec{k}}(\vec{\bm{x}}).
\end{align*}
Next, for the second term, since
\[
	\calD_1 \partial_{\bm{c}}(\bm{x})^k = \partial_{\bm{c}}(\bm{x})^k \calD_1 -k \partial_{\bm{c}}(\bm{x})^k
\]
is known from~\cite[Lemma 3.3]{RoehrigZwegers2022}, and $f$ is a spherical polynomial of degree $d$, 
\begin{align*}
	\calD_m G_{\vec{\sigma}, \vec{k}}(\vec{\bm{x}}) &= \sum_{j=1}^m \calD_1^{(j)} G_{\vec{\sigma}, \vec{k}}(\vec{\bm{x}}) = \left(\prod_{i=1}^m \partial_{\bm{c}_i^{(\sigma_i)}}(\bm{x}_i)^{k_i}\right) \sum_{j=1}^m (\calD_1^{(j)} - k_j) f(\vec{\bm{x}})\\
	&= (d - k) G_{\vec{\sigma}, \vec{k}}(\vec{\bm{x}}).
\end{align*}
Finally, we can verify that the third term equals
\begin{align*}
	-\frac{1}{2\pi} \sum_{j=1}^m (k_j+1) F_{\vec{\sigma}, \vec{k}+\vec{1}_j}(\vec{\bm{x}}) G_{\vec{\sigma}, \vec{k} + \vec{1}_j}(\vec{\bm{x}}),
\end{align*}
where $\vec{1}_j \in \Z^m$ is the vector whose $j$-th component is $1$ and all other components are $0$. 
	
Therefore, for each $\vec{\sigma}$, we obtain
\begin{align*}
	&\calD_m \sum_{k=0}^d \frac{(-1)^k}{(4\pi)^k} \sum_{\substack{\vec{k} \in (\Z_{\ge 0})^m \\ |\vec{k}|= k}} F_{\vec{\sigma}, \vec{k}}(\vec{\bm{x}}) G_{\vec{\sigma}, \vec{k}}(\vec{\bm{x}})\\
	&= \sum_{k=0}^d \frac{(-1)^k}{(4\pi)^k} \sum_{\substack{\vec{k} \in (\Z_{\ge 0})^m \\ |\vec{k}| = k}} \left((d - 2k) F_{\vec{\sigma}, \vec{k}}(\vec{\bm{x}}) G_{\vec{\sigma}, \vec{k}}(\vec{\bm{x}}) -\frac{1}{2\pi} \sum_{j=1}^m (k_j+1) F_{\vec{\sigma}, \vec{k}+\vec{1}_j}(\vec{\bm{x}}) G_{\vec{\sigma}, \vec{k} + \vec{1}_j}(\vec{\bm{x}}) \right).
\end{align*}
For the second sum, by changing variables by replacing $k_j+1$ with $k_j$ for each $j$, it becomes
\begin{align*}
	&= \sum_{k=0}^d \frac{(-1)^k}{(4\pi)^k} (d - 2k) \sum_{\substack{\vec{k} \in (\Z_{\ge 0})^m \\ |\vec{k}| = k}} F_{\vec{\sigma}, \vec{k}}(\vec{\bm{x}}) G_{\vec{\sigma}, \vec{k}}(\vec{\bm{x}}) + 2 \sum_{j=1}^m \sum_{k=0}^d \frac{(-1)^{k+1}}{(4\pi)^{k+1}} \sum_{\substack{\vec{k} \in (\Z_{\ge 0})^m \\ |\vec{k}| = k+1}} k_j F_{\vec{\sigma}, \vec{k}}(\vec{\bm{x}}) G_{\vec{\sigma}, \vec{k}}(\vec{\bm{x}}).
\end{align*}
Finally, by replacing $k+1$ with $k$ in the second term, the $-2k$ from the first term and the $2k$ from the second term cancel out, yielding the desired result.
\end{proof}

Thus, Vign\'{e}ras' criterion (\cref{fact:Vig}) implies the following result, providing non-holomorphic modular forms.

\begin{theorem}\label{thm:indefinite-theta}
	Under the same notation as in \cref{def:deg-m-p}, we assume that $f(\vec{\bm{x}})$ is a spherical polynomial of degree $d$. For $\vec{\bm{a}}, \vec{\bm{b}} \in \R^{2m}$, the theta function defined by
	\[
		\theta_{\vec{\bm{a}}, \vec{\bm{b}}}^{\vec{\bm{c}}_0, \vec{\bm{c}}_1}[f] (\tau) = v^{-d/2} \sum_{\vec{\bm{x}} \in \vec{\bm{a}} + \Z^{2m}} p^{\vec{\bm{c}}_0, \vec{\bm{c}}_1}[f](\vec{\bm{x}} \sqrt{v}) q^{Q_m(\vec{\bm{x}})} e^{2\pi i B_m(\vec{\bm{x}}, \vec{\bm{b}})}
	\]
	converges absolutely and satisfies
	\[
		\theta_{\vec{\bm{a}}, \vec{\bm{b}}}^{\vec{\bm{c}}_0, \vec{\bm{c}}_1}[f] \left(-\frac{1}{\tau}\right) = (-\tau)^{m+d} e^{2\pi iB_m(\vec{\bm{a}}, \vec{\bm{b}})} \theta_{-\vec{\bm{b}}, \vec{\bm{a}}}^{\vec{\bm{c}}_0, \vec{\bm{c}}_1}[f] (\tau).
	\]
\end{theorem}

\subsection{The modularity of $\KW_m(\tau)$}\label{modularity-KW}

As an application of \cref{thm:indefinite-theta}, we derive the modular transformation laws of the function $\KW_m(\tau)$. First, we show that the homogeneous polynomial $V_m(\vec{\bm{x}})$ of degree $2m^2$ defined in~\cref{def:Vm} is spherical.

\begin{lemma}
	For any $m \ge 1$, we have $\Delta_m V_m(\vec{\bm{x}}) = 0$.
\end{lemma}

\begin{proof}
Let
\[
	A_m(\vec{x}) = \prod_{1 \le i < j \le m} (x_i^2 - x_j^2), \quad B_m(\vec{x}) = \prod_{j=1}^m x_j
\]
for $\vec{x} = (x_1, \dots, x_m) \in \R^m$. It suffices to show that 
\[
	\sum_{j=1}^m \frac{\partial^2}{\partial x_j^2} \bigg( A_m(\vec{x})B_m(\vec{x}) \bigg) = \sum_{j=1}^m\frac{\partial^2 A_m}{\partial x_j^2}(\vec{x}) B_m(\vec{x}) + 2\sum_{j=1}^m \frac{\partial A_m}{\partial x_j} (\vec{x}) \frac{\partial B_m}{\partial x_j} (\vec{x}) + \sum_{j=1}^mA_m(\vec{x}) \frac{\partial^2 B_m}{\partial x_j^2}(\vec{x})
\]
vanishes. The third sum is $0$ since $B_m(\vec{x})$ has degree $1$ in each $x_j$. The second sum yields 
\begin{align*}
	\sum_{j=1}^m \frac{\partial A_m}{\partial x_j} (\vec{x}) \frac{\partial B_m}{\partial x_j} (\vec{x}) &= \sum_{j=1}^m 2x_j A_m(\vec{x}) \left(\sum_{\substack{1 \le k \le m \\ k \neq j}} \frac{1}{x_j^2 - x_k^2}\right) \cdot \prod_{\substack{1 \le k \le m \\ k \neq j}} x_k\\
		&= 2A_m(\vec{x}) B_m(\vec{x}) \sum_{\substack{1 \le j, k \le m \\ j \neq k}} \frac{1}{x_j^2 - x_k^2},
\end{align*}
while the double sum becomes zero due to its symmetry. Finally, regarding the first sum, we further differentiate the expression for the derivative of $A_m(\vec{x})$ computed above to obtain
\[
	\sum_{j=1}^m \frac{\partial^2 A_m}{\partial x_j^2} (\vec{x}) = A_m(\vec{x}) \sum_{\substack{1 \le j, k, l \le m \\ j \neq k, j \neq l, k \neq l}} \frac{4 x_j^2}{(x_j^2 - x_k^2)(x_j^2 - x_l^2)}.
\]
Since $f(x,y,z) = x/(x-y)(x-z)$ satisfies $f(x,y,z) + f(y,z,x) + f(z,x,y) = 0$, this also equals $0$.
\end{proof}

By taking suitable limits of the non-holomorphic modular forms constructed in~\cref{thm:indefinite-theta}, we obtain holomorphic modular forms associated with indefinite quadratic forms. As an example, we aim to realize Kac--Wakimoto's series $\KW_m(\tau)$ as a holomorphic modular form.

\begin{theorem}\label{thm:KW-indefinite}
For $t > 0$, we set $\vec{\bm{c}}_0(t) = (\smat{1 \\ 1+t}, \dots, \smat{1 \\ 1+t}), \vec{\bm{c}}_1(t) = (\smat{-1 \\ 1+t}, \dots, \smat{-1 \\ 1+t}) \in \mathcal{C}^m$ and let $f(\vec{\bm{x}})$ be a spherical polynomial. 
For a vector $\vec{\bm{a}} = (\bm{a}_1, \dots, \bm{a}_m) \in \R^{2m}$ satisfying $B_1(\smat{1 \\ 1}, \bm{a}_j)$, $B_1(\smat{-1 \\ 1}, \bm{a}_j) \not\in \Z$ for any $1 \le j \le m$, and any vector $\vec{\bm{b}} \in \R^{2m}$, we have
	\begin{align}\label{eq:theta-limit}
		\lim_{t \to 0} \theta_{\vec{\bm{a}}, \vec{\bm{b}}}^{\vec{\bm{c}}_0(t), \vec{\bm{c}}_1(t)} [f] (\tau) &= \sum_{\vec{\bm{x}} \in \vec{\bm{a}} + \Z^{2m}} \prod_{j=1}^m \bigg(\sgn(x_j-y_j) - \sgn(-x_j -y_j)\bigg) \times f(\vec{\bm{x}}) q^{Q_m(\vec{\bm{x}})} e^{2\pi iB_m(\vec{\bm{x}}, \vec{\bm{b}})}.
	\end{align}
	In particular, for $\vec{\bm{a}} = (\smat{0 \\ 1/2}, \dots, \smat{0 \\ 1/2}), \vec{\bm{b}} = (\smat{1/2 \\ 1/2}, \dots, \smat{1/2 \\ 1/2})$ and $f(\vec{\bm{x}}) = V_m(\vec{\bm{x}})$, we have
	\begin{align}\label{eq:limit-KW}
		\lim_{t \to 0} \theta_{\vec{\bm{a}}, \vec{\bm{b}}}^{\vec{\bm{c}}_0(t), \vec{\bm{c}}_1(t)} [V_m] (\tau) = (4i)^m m! \bigg(\prod_{j=1}^m (2j-1)!\bigg)^2 \KW_m(\tau).
	\end{align}
\end{theorem}

\begin{proof}
The overall approach of the proof follows that of Roehrig--Zwegers~\cite[Theorem 2.4]{RoehrigZwegers2022-pre}. First, to verify the convergence of the right-hand side of~\eqref{eq:theta-limit}, recall that in~\cite[Lemma 3.1]{RoehrigZwegers2022-pre}, the absolute convergence of 
\begin{align}\label{eq:abs-conv-1}
	\sum_{\bm{x} \in \bm{a} + \Z^2} \bigg(\sgn(x-y) - \sgn(-x-y)\bigg) P(\bm{x}) q^{Q_1(\bm{x})} e^{2\pi iB_1(\bm{x}, \bm{b})}
\end{align}
is established for any polynomial $P(\bm{x})$ and for $\bm{a}, \bm{b} \in \R^2$ satisfying $B_1(\smat{1 \\ 1}, \bm{a}), B_1(\smat{-1 \\ 1}, \bm{a}) \not\in \Z$. By expressing
\begin{align}\label{eq:f-poly-exp}
	f(\vec{\bm{x}}) = \sum_{\vec{e}, \vec{f} \in (\Z_{\ge 0})^m} c_{\vec{e}, \vec{f}} \prod_{j=1}^m x_j^{e_j} y_j^{f_j}
\end{align}
for some constant coefficients $c_{\vec{e}, \vec{f}}$, the right-hand side of~\eqref{eq:theta-limit} becomes
\[
	\sum_{\vec{e}, \vec{f} \in (\Z_{\ge 0})^m} c_{\vec{e}, \vec{f}} \prod_{j=1}^m \left( \sum_{\bm{x}_j \in \bm{a}_j + \Z^2} \bigg(\sgn(x_j-y_j) - \sgn(-x_j -y_j)\bigg) x_j^{e_j} y_j^{f_j} q^{Q_1(\bm{x}_j)} e^{2\pi iB_1(\bm{x}_j, \bm{b}_j)} \right).
\]
By applying the above convergence result to each sum over $\bm{x}_j$, absolute convergence follows for the entire expression.

Next, we proceed with the calculation of the limit for the main result. By splitting the sum defining $p^{\vec{\bm{c}}_0(t), \vec{\bm{c}}_1(t)}[f](\vec{\bm{x}})$ into the case $k=0$ and the other terms, the summand of $\theta_{\vec{\bm{a}}, \vec{\bm{b}}}^{\vec{\bm{c}}_0(t), \vec{\bm{c}}_1(t)}[f](\tau)$ can be divided as follows:
\begin{align}\label{eq:summand-theta-0}
\begin{split}
	&v^{-d/2} p^{\vec{\bm{c}}_0(t), \vec{\bm{c}}_1(t)}[f](\vec{\bm{x}} \sqrt{v}) e^{-2\pi Q_m(\vec{\bm{x}}) v}\\
	&= \prod_{j=1}^m H_j(\bm{x}_j, t) \cdot f(\vec{\bm{x}}) e^{-2\pi Q_m(\vec{\bm{x}}) v} + v^{-d/2} \sum_{k=1}^d \frac{(-1)^k}{(4\pi)^k} \sum_{\substack{\vec{k} \in (\Z_{\ge 0})^m \\ |\vec{k}| = k}} S_{\vec{k}}(\vec{\bm{x}} \sqrt{v}, t),
\end{split}
\end{align}
where we set
\begin{align*}
	H_j(\bm{x}_j, t) \coloneqq E\left(\frac{B_1(\bm{c}_j^{(0)}(t), \bm{x}_j \sqrt{v})}{\sqrt{-Q_1(\bm{c}_j^{(0)}(t))}} \right) - E\left(\frac{B_1(\bm{c}_j^{(1)}(t), \bm{x}_j \sqrt{v})}{\sqrt{-Q_1(\bm{c}_j^{(1)}(t))}} \right).
\end{align*}
Note that $S_{\vec{k}}(\bm{x}, t)$ is the same as $S_{\vec{k}}(\vec{\bm{x}})$ defined in~\eqref{eq:def-Sk}, but since it is now defined with respect to $\vec{\bm{c}}_0(t), \vec{\bm{c}}_1(t)$, the notation explicitly shows $t$.  
Our goal is to show that only the first term contributes, while the contributions from the remaining terms vanish. For convenience, we define
\[
	H_j(\bm{x}_j, 0) \coloneqq \sgn(x_j-y_j) - \sgn(-x_j-y_j).
\]
It holds that $\lim_{t \to 0} H_j(\bm{x}_j, t) = H_j(\bm{x}_j, 0)$. 

For the first term, we see that
\begin{align*}
	&\prod_{j=1}^m H_j(\bm{x}_j, t) - \prod_{j=1}^m H_j(\bm{x}_j,0)\\
	&= \sum_{k=1}^m \left(\prod_{1 \le j < k} H_j(\bm{x}_j, 0) \times (H_k(\bm{x}_k, t) - H_k(\bm{x}_k, 0)) \times \prod_{k < j \le m} H_j(\bm{x}_j, t) \right).
\end{align*}
For each $k$, we aim to show that the series
\begin{align}\label{eq:contribution}
	\sum_{\vec{\bm{x}} \in \vec{\bm{a}} + \Z^{2m}} \left(\prod_{1 \le j < k} H_j(\bm{x}_j, 0) \times (H_k(\bm{x}_k, t) - H_k(\bm{x}_k, 0)) \times \prod_{k < j \le m} H_j(\bm{x}_j, t)\right) f(\vec{\bm{x}}) q^{Q_m(\vec{\bm{x}})} e^{2\pi i B_m(\vec{\bm{x}}, \vec{\bm{b}})}
\end{align}
converges to $0$ as $t \to 0$. In this case as well, similar to~\eqref{eq:f-poly-exp}, we decompose $f(\vec{\bm{x}})$ into monomials, and the discussion can be reduced to the case for each $j$. 
For $1 \le j < k$, the absolute convergence was already discussed in~\eqref{eq:abs-conv-1}. 
For $j = k$, in the proof of~\cite[Theorem 2.4]{RoehrigZwegers2022-pre}, it is known that
\begin{align}\label{eq:Hk-diff}
	\lim_{t \to 0} \sum_{\bm{x}_k \in \bm{a}_k + \Z^2} (H_k(\bm{x}_k, t) - H_k(\bm{x}_k,0)) P(\bm{x}_k) q^{Q_1(\bm{x}_k)} e^{2\pi i B_1(\bm{x}_k, \bm{b}_k)} = 0
\end{align}
for any polynomial $P(\bm{x})$. Finally, for $k < j \le m$,  we define the complementary error function $\beta(z)$ by
\[
	\beta(z) \coloneqq \int_z^\infty u^{-1/2} e^{-\pi u} \dd u.
\]
As shown in~\cite[Lemma 1.7]{Zwegers2002}, it is known that
\begin{align}\label{eq:Errors}
	E(z) = \sgn(z) - \sgn(z) \beta(z^2).
\end{align}
We substitute this into $H_j(\bm{x}_j, t)$, and then separate the terms containing $\sgn(z)$ from those containing $\sgn(z) \beta(z^2)$, evaluating each separately. First, the terms containing $\sgn(z)$ correspond to
\begin{align*}
	\sum_{\bm{x}_j \in \bm{a}_j + \Z^2} \bigg(\sgn(x_j - (1+t)y_j) - \sgn(-x_j - (1+t)y_j) \bigg) P(\bm{x}_j) q^{Q_1(\bm{x}_j)} e^{2\pi i B_1(\bm{x}_j, \bm{b}_j)}.
\end{align*}
Since $\{(x,y) \in \R^2 \mid \sgn(x- (1+t)y) \neq \sgn(-x-(1+t)y)\} \subset \{(x,y) \in \R^2 \mid \sgn(x- y) \neq \sgn(-x-y)\}$, this sum is a partial sum of the absolutely convergent series in~\eqref{eq:abs-conv-1}. 
Therefore,  this sum also converges absolutely and uniformly in $t>0$. 
On the other hand, the terms containing $\sgn(z) \beta(z^2)$ correspond to
\[
	\sum_{\bm{x}_j \in \bm{a}_j + \Z^2} \sgn(x_j - (1+t)y_j) \beta\left(\frac{B_1(\bm{c}_j^{(0)}(t), \bm{x}_j \sqrt{v})^2}{-Q_1(\bm{c}_j^{(0)}(t))}\right) P(\bm{x}_j) q^{Q_1(\bm{x}_j)} e^{2\pi i B_1(\bm{x}_j, \bm{b}_j)}
\]
and one more term for $\bm{c}_j^{(1)}(t)$. 
This has been shown in the proof of~\cite[Theorem 2.4]{RoehrigZwegers2022-pre} that it converges to $0$ as $t \to 0$. 
Therefore, by the contribution of~\eqref{eq:Hk-diff}, it follows that \eqref{eq:contribution} converges to $0$ as $t \to 0$.

Finally, we verify that the contribution from the remaining terms in~\eqref{eq:summand-theta-0} vanishes for each $\vec{k}$, that is,
\begin{align*}
	\lim_{t \to 0} \sum_{\vec{\bm{x}} \in \vec{\bm{a}} + \Z^{2m}} S_{\vec{k}}(\vec{\bm{x}} \sqrt{v}, t) e^{2\pi i u Q_m(\vec{\bm{x}}) + 2\pi i B_m(\vec{\bm{x}}, \vec{\bm{b}})} = 0.
\end{align*}
Using the explicit formula for $S_{\vec{k}}(\vec{\bm{x}} \sqrt{v}, t)$ from~\eqref{eq:Sk-explicit} and decomposing the polynomial part into monomials as before, we evaluate the sum for each $j$. 
For $j \in I(\vec{k})$, the evaluation was already done in the case of the first term for $k < j \le m$, and it follows that the sum converges absolutely and uniformly in $t > 0$. 
For $j \in I(\vec{k})^c$, it has been shown in the proof of~\cite[Theorem 2.4]{RoehrigZwegers2022-pre} that the sum converges to $0$. 
Since $I(\vec{k})^c \neq \emptyset$, the entire sum converges to $0$ as $t \to 0$. This concludes the proof.

The final result in~\eqref{eq:limit-KW} directly follows from the expression for $\KW_m(\tau)$ given in~\cref{thm:KW-rewrite}.
\end{proof}

\begin{corollary}\label{cor:modular-KW}
	The series $\KW_m(\tau)$ satisfies the following modular transformation laws on $\Gamma(2)$.
	\begin{align*}
		\mathrm{KW}_m(\tau+2) &= (-i)^m \mathrm{KW}_m(\tau),\\
		\mathrm{KW}_m \left(\frac{\tau}{2\tau+1}\right) &= (2\tau+1)^{m(2m+1)} \mathrm{KW}_m(\tau).
	\end{align*}
\end{corollary}

\begin{proof}
We take $\vec{\bm{c}}_0(t), \vec{\bm{c}}_1(t), \vec{\bm{a}}, \vec{\bm{b}}$ as in \cref{thm:KW-indefinite}. Since $e^{2\pi i(x^2-y^2)} = -i$ for any $x \in \Z$ and $y \in 1/2 + \Z$, it follows that
\begin{align*}
	\theta_{\vec{\bm{a}}, \vec{\bm{b}}}^{\vec{\bm{c}}_0(t), \vec{\bm{c}}_1(t)} [V_m] (\tau+2) = (-i)^m \theta_{\vec{\bm{a}}, \vec{\bm{b}}}^{\vec{\bm{c}}_0(t), \vec{\bm{c}}_1(t)} [V_m] (\tau).
\end{align*}
Taking the limit as $t \to 0$ yields the first result. 
For the second result, \cref{thm:indefinite-theta} implies
\begin{align*}
	\theta_{\vec{\bm{a}}, \vec{\bm{b}}}^{\vec{\bm{c}}_0(t), \vec{\bm{c}}_1(t)} [V_m]\left(\frac{\tau}{2\tau+1}\right) &= \theta_{\vec{\bm{a}}, \vec{\bm{b}}}^{\vec{\bm{c}}_0(t), \vec{\bm{c}}_1(t)} [V_m] \left(-\frac{1}{-2-\frac{1}{\tau}}\right)\\
		&= \left(2 + \frac{1}{\tau}\right)^{m(2m+1)} e^{2\pi i \frac{-m}{4}} \theta_{-\vec{\bm{b}}, \vec{\bm{a}}}^{\vec{\bm{c}}_0(t), \vec{\bm{c}}_1(t)} [V_m] \left(-2 - \frac{1}{\tau}\right).
\end{align*}
Since $e^{2\pi i(x^2-y^2)} = 1$ for any $x, y \in 1/2+ \Z$, it becomes
\begin{align*}
	&= (-i)^m \left(2 + \frac{1}{\tau}\right)^{m(2m+1)} \theta_{\vec{\bm{b}}, \vec{\bm{a}}}^{\vec{\bm{c}}_0(t), \vec{\bm{c}}_1(t)} [V_m] \left(- \frac{1}{\tau}\right)\\
	&= (-i)^m \left(2 + \frac{1}{\tau}\right)^{m(2m+1)} (-\tau)^{m(2m+1)} e^{2\pi i \frac{-m}{4}} \theta_{-\vec{\bm{a}}, \vec{\bm{b}}}^{\vec{\bm{c}}_0(t), \vec{\bm{c}}_1(t)} [V_m](\tau)\\
	&= (2\tau+1)^{m(2m+1)} \theta_{\vec{\bm{a}}, \vec{\bm{b}}}^{\vec{\bm{c}}_0(t), \vec{\bm{c}}_1(t)} [V_m](\tau).
\end{align*}
Taking the limit as $t \to 0$ concludes the proof.
\end{proof}
	
\begin{corollary}\label{cor:cusps-KW}
	The behaviors of $\mathrm{KW}_m(\tau)$ at the cusps $i\infty, 0, 1$ are given by
	\begin{align*}
		\mathrm{KW}_m(\tau) &= q^{\frac{m(2m+1)}{8}} + \cdots,\\
		(-i\tau)^{-m(2m+1)} \mathrm{KW}_m \left(-\frac{1}{\tau}\right) &= O(1),\\
		(-i\tau)^{-m(2m+1)} \mathrm{KW}_m \left(\frac{\tau-1}{\tau}\right) &= O(q^{\frac{m(2m+1)}{8}}).
	\end{align*}
\end{corollary}

\begin{proof}
The definition of $\KW_m(\tau)$ in \cref{thm:KW-rewrite} implies that the lowest degree in its $q$-series expansion is given by
    \[
	\ell = \min\left\{ \frac{1}{2} \sum_{j=1}^m (x_j^2 - y_j^2) \, \middle|\, x_j \in \Z, y_j \in 
	\frac{1}{2} + \Z, x_j^2 > y_j^2, x_i^2 \neq x_j^2, y_i^2 \neq y_j^2, (i \neq j) \right\}.
    \]
Note that $V_m(\vec{\bm{x}})$ vanishes when there exists a pair $i \neq j$ such that $x_i^2 = x_j^2$ or $y_i^2 = y_j^2$. 
The minimum is attained when $(x_j, y_j) = (j, j-1/2)$, giving
    \[
	\ell = \frac{1}{2} \sum_{j=1}^m (j^2 - (j-1/2)^2) = \frac{m(2m+1)}{8}.
    \]
To compute the coefficient of $q^{m(2m+1)/8}$, we use the identity
    \[
	V_m \left(\pmat{1 \\ 1/2}, \pmat{2 \\ 3/2}, \dots, \pmat{m \\ m-1/2} \right) = \frac{1}{2^m} \left(\prod_{j=1}^m (2j-1)! \right)^2,
    \]
which follows by induction on $m$. 
From the expression of $\KW_m(\tau)$ in \eqref{eq:KW-cone}, the coefficient of $q^{m(2m+1)/8}$ is calculated as
\begin{align*}
	\frac{2^m}{m! \left(\prod_{j=1}^m (2j-1)!\right)^2} V_m \left(\pmat{1 \\ 1/2}, \pmat{2 \\ 3/2}, \dots, \pmat{m \\ m-1/2} \right) \times m! = 1,
\end{align*}
where $m!$ arises as a factor from the counting of permutations of the $\bm{x}_j$'s. This completes the proof of the first result.

For the second result, we take $\vec{\bm{c}}_0(t), \vec{\bm{c}}_1(t), \vec{\bm{a}}, \vec{\bm{b}}$ as in \cref{thm:KW-indefinite} and replace $\vec{\bm{b}}$ with $\vec{\bm{b}} + \vec{\bm{\ep}}$, where $\vec{\bm{\ep}} = (\smat{\ep \\ 0}, \dots, \smat{\ep \\ 0})$ ($0 < \ep < 1$). 
Applying \cref{thm:indefinite-theta} and \cref{thm:KW-indefinite} yields
    \[
	(-i \tau)^{-m(2m+1)} \lim_{t \to 0} 
	\theta_{\vec{\bm{a}}, \vec{\bm{b}}+\vec{\bm{\ep}}}^{\vec{\bm{c}}_0(t), \vec{\bm{c}}_1(t)}[V_m] 
	\left(-\frac{1}{\tau}\right) 
	= (-i)^{m(2m+1)} e^{2\pi iB_m(\vec{\bm{a}}, \vec{\bm{b}}+\vec{\bm{\ep}})} 
	\lim_{t \to 0} \theta_{-\vec{\bm{b}}-\vec{\bm{\ep}}, \vec{\bm{a}}}^{\vec{\bm{c}}_0(t), \vec{\bm{c}}_1(t)}[V_m] (\tau).
    \]
As $\ep \to 0$, the left-hand side converges to $(-i\tau)^{-m(2m+1)} \KW_m(-1/\tau)$. On the other hand, 
\begin{align}\label{eq:ep-shift}
\begin{split}
	&\lim_{t \to 0} \theta_{-\vec{\bm{b}}-\vec{\bm{\ep}}, \vec{\bm{a}}}^{\vec{\bm{c}}_0(t), \vec{\bm{c}}_1(t)}[V_m] (\tau)\\
	&= \sum_{\vec{\bm{x}} \in -\vec{\bm{b}} + \Z^{2m}} \prod_{j=1}^m \bigg(\sgn(x_j-y_j - \ep) - \sgn(-x_j -y_j + \ep)\bigg) \times V_m(\vec{\bm{x}} - \vec{\bm{\ep}}) q^{Q_m(\vec{\bm{x}} - \vec{\bm{\ep}})} e^{2\pi iB_m(\vec{\bm{x}} - \vec{\bm{\ep}}, \vec{\bm{a}})}.
\end{split}
\end{align}
Expressing $V_m(\vec{\bm{x}} - \vec{\bm{\ep}})$ as
\[
	V_m(\vec{\bm{x}} - \vec{\bm{\ep}}) = \sum_{\vec{e}, \vec{f} \in (\Z_{\ge 0})^m} p_{\vec{e}, \vec{f}}(\ep) \prod_{j=1}^m x_j^{e_j} y_j^{f_j}
\]
for some polynomials $p_{\vec{e}, \vec{f}}(\ep) \in \Z[\varepsilon]$, the right-hand side of \eqref{eq:ep-shift} becomes
\begin{align}\label{eq:V-expand}
	\sum_{\vec{e}, \vec{f} \in (\Z_{\ge 0})^m} p_{\vec{e}, \vec{f}}(\ep) \prod_{j=1}^m \left(\sum_{x_j, y_j \in 1/2 + \Z} \bigg(\sgn(x_j-y_j - \ep) - \sgn(-x_j -y_j + \ep)\bigg) x_j^{e_j} y_j^{f_j} q^{\frac{(x_j - \ep)^2 - y_j^2}{2}} e^{-\pi i y_j} \right).
\end{align}
To analyze the inner sum as $\ep \to 0$ for each $j$, we split it into two cases: when $x_j^2 - y_j^2 \neq 0$ and when $x_j^2 - y_j^2 = 0$. 
For the former case,  we can interchange the order of limit and summation,  yielding
\begin{align}\label{eq:limit-1}
\begin{split}
	&\lim_{\ep \to 0} \sum_{\substack{x_j, y_j \in 1/2 + \Z \\ x_j^2 - y_j^2 \neq 0}} \bigg(\sgn(x_j-y_j - \ep) - \sgn(-x_j -y_j + \ep)\bigg) x_j^{e_j} y_j^{f_j} q^{\frac{(x_j - \ep)^2 - y_j^2}{2}} e^{-\pi i y_j}\\
	&= \sum_{\substack{x_j, y_j \in 1/2 + \Z \\ x_j^2 - y_j^2 \neq 0}}  \bigg(\sgn(x_j-y_j) - \sgn(-x_j -y_j)\bigg) x_j^{e_j} y_j^{f_j} q^{\frac{x_j^2 - y_j^2}{2}} e^{-\pi i y_j} = O(q).
\end{split}
\end{align}
For the latter case, note that when $x_j > 0$, the difference $\sgn(x_j - y_j - \ep) - \sgn(-x_j - y_j + \ep)$ vanishes. Thus,
\begin{align*}
	&\lim_{\ep \to 0} \sum_{\substack{x_j, y_j \in 1/2 + \Z \\ x_j^2 - y_j^2 = 0}} \bigg(\sgn(x_j-y_j - \ep) - \sgn(-x_j -y_j + \ep)\bigg) x_j^{e_j} y_j^{f_j} q^{\frac{(x_j - \ep)^2 - y_j^2}{2}} e^{-\pi i y_j}\\
	&= -2 \lim_{\ep \to 0} \left( \sum_{x_j \in 1/2 + \Z_{<0}} x_j^{e_j} x_j^{f_j} q^{\frac{\ep(\ep-2x_j)}{2}} e^{-\pi ix_j} + \sum_{x_j \in 1/2 + \Z_{<0}} x_j^{e_j} (-x_j)^{f_j} q^{\frac{\ep(\ep-2x_j)}{2}} e^{\pi ix_j} \right).
\end{align*}
Setting $x_j = 1/2-n$ yields
\begin{align}\label{eq:ep-shift-limit}
	&= -2 (-i + (-1)^{f_j} i) \lim_{\ep \to 0} q^{\frac{\ep(\ep-1)}{2}} \sum_{n=1}^\infty (-1)^n (1/2-n)^{e_j+f_j} q^{\ep n}.
\end{align}
Since it is known that there exists $E_k(x) \in \Z[x]$ satisfying
\[
	\sum_{n=1}^\infty (-1)^n n^k q^{\ep n} = \frac{E_k(q^\ep)}{(1+q^\ep)^{k+1}},
\]
the limit in \eqref{eq:ep-shift-limit} converges to a constant. Combining this with \eqref{eq:limit-1} shows that the limit of \eqref{eq:V-expand} as $\ep \to 0$ becomes $O(1)$ as $q \to 0$.

Finally, for the third result, note that $e^{2\pi i \frac{x^2 - y^2}{2}} = e^{-\pi i/4} (-1)^x$ for any $x \in \Z, y \in 1/2 + \Z$. Then,
\begin{align*}
	(-i\tau)^{-m(2m+1)} \theta_{\vec{\bm{a}}, \vec{\bm{b}}}^{\vec{\bm{c}}_0(t), \vec{\bm{c}}_1(t)}[V_m] \left(\frac{\tau-1}{\tau}\right) &= (-i\tau)^{-m(2m+1)} e^{-\frac{\pi i m}{4}} \theta_{\vec{\bm{a}}, \vec{\bm{a}}}^{\vec{\bm{c}}_0(t), \vec{\bm{c}}_1(t)}[V_m] \left(-\frac{1}{\tau}\right)\\
	&= (-i\tau)^{-m(2m+1)} e^{-\frac{\pi i m}{4}} (-\tau)^{m(2m+1)} e^{2\pi i \frac{-m}{4}} \theta_{-\vec{\bm{a}}, \vec{\bm{a}}}^{\vec{\bm{c}}_0(t), \vec{\bm{c}}_1(t)} [V_m] (\tau)\\
	&= e^{-\frac{\pi im}{4}} \theta_{\vec{\bm{a}}, \vec{\bm{a}}}^{\vec{\bm{c}}_0(t), \vec{\bm{c}}_1(t)}[V_m] (\tau).
\end{align*}
Taking the limit as $t \to 0$ and comparing it with the first result yields $O(q^{m(2m+1)/8})$.
\end{proof}

In conclusion, we provide an alternative proof of \cref{thm:Kac-Wakimoto-Conj} by applying the theory of modular forms.

\begin{proof}[Proof of \cref{thm:Kac-Wakimoto-Conj}]
We show the identity~\eqref{eq:KW-identity}. Comparing \cref{cor:modular-KW} with \cref{lem:theta-triangle}, we have that the quotient
\[
	\frac{\KW_m(\tau)}{\theta_\triangle(\tau)^{2m(2m+1)}}
\]
is invariant under the action of $\Gamma(2)$. In particular, since $\theta_\triangle(\tau)$ has no zeros on $\bbH$, it defines a holomorphic function on $\bbH/\Gamma(2)$. Furthermore, by comparing \cref{cor:cusps-KW} with \cref{lem:theta-triangle}, the limit at each cusp $i\infty, 0, 1$ is bounded, with the limit at $i\infty$ being equal to $1$. Since the genus of $\bbH/\Gamma(2)$ is zero, Liouville's theorem implies that this must be constant, and thus equals $1$.
\end{proof}

\bibliographystyle{amsalpha}
\bibliography{References} 

\end{document}